\documentclass[leqno]{siamltex}

\usepackage{amssymb}
\usepackage{amsmath}
\usepackage{graphicx}

\usepackage[letterpaper=false,colorlinks=true,
linkcolor=red,filecolor=green,citecolor=red,
pdfpagemode=None]{hyperref}

\graphicspath{{Images/}}

\newcommand{\blue}[1]{\textcolor{\blue}{#1}}

\newcommand{\C}{\ensuremath{\mathbf{C}}}
\newcommand{\Rn}[1][m]{\ensuremath{\mathbf{R}^{#1}}}
\DeclareMathOperator{\vecop}{vec}
\newcommand{\set}[1]{\ensuremath{\left\{ #1 \right\}}}

\newtheorem{theo}{Theorem}[section]

\newtheorem{lem}[theo]{Lemma}

\newtheorem{remark}[theorem]{Remark}

\newcommand{\bC}{\mathbf{C}}
\newcommand{\bCbar}{\mathbf{\bar{C}}}
\newcommand{\bT}{\mathbf{T}}
\newcommand{\bQ}{\mathbf{Q}}

\newcommand{\Lop}{\ensuremath{\mathcal{L}}}

\newcommand{\MK}[1]{#1}

\title{Preconditioning a coupled model \\ for reactive transport in
  porous media \thanks{This work was partially supported by GNR MoMaS, CNRS-2439 (PACEN/CNRS, ANDRA, BRGM, CEA, EDF, IRSN), and by the
  Hydrinv(EuroMediterrannean 3+3) project}
}

\author{Laila Amir\footnotemark[2]
\and Michel Kern\footnotemark[3]
}

\begin{document}
\maketitle

\renewcommand{\thefootnote}{\fnsymbol{footnote}}
\footnotetext[2]{Facult\'e des sciences et techniques, UCAM, Marrakech,
  Morocco, \texttt{l.amir@uca.ma}}
\footnotetext[3]{Centre de recherche Inria de Paris, 2 rue Simone Iff,
  75012 Paris, France, \texttt{michel.kern@inria.fr}}
\renewcommand{\thefootnote}{\arabic{footnote}}


\begin{abstract}
We study numerical methods for solving reactive transport problems in
porous media that allow a separation of transport and chemistry at the
software level, while keeping a tight numerical coupling between both
subsystems. 
We give a formulation that
  eliminates the local chemical concentrations and keeps the total
  concentrations as unknowns, then recall how each individual
  subsystem can be solved. The coupled system is solved by a
  Newton--Krylov method.
The block structure of the model is exploited both 
at the nonlinear level, by eliminating some unknowns, and at the linear level
 by using block Gauss-Seidel or block Jacobi preconditioning.
The methods are applied to a 1D case of the MoMaS benchmark. 
\end{abstract}

\pagestyle{myheadings}\thispagestyle{plain}
\markboth{L. Amir and M. Kern}{Preconditioning for reactive transport}

\section{Introduction}
\label{sec:intro}

Reactive transport in porous media studies the coupling of reacting chemical
species with mass transport in the subsurface,
see~\cite[sec. 7.9]{bear:2010} for an introduction,
or~\cite{AppelPost:05,bethke:08} for more comprehensive
references. It plays an important role in several applications
when modeling subsurface
flow~\cite{Steefel2015,steeflicht:05,zhangyehparker:reactransbook} :  
\begin{itemize}
\item chemical trapping is one of the
mechanism by which the safety of geological sequestration of carbon
dioxide in deep saline aquifers can be
ensured~\cite{audigane:hal-00564444,fan2012fully,12985679,trentymichelECMOR:05}
\item nuclear waste storage is based on a multiple barrier concept so
  as to delay the arrival of radionuclides in the bio-sphere. The
  concrete barrier that seals the repository may be attacked by
  oxidized compounds, and again its safety needs to be assessed~
  \cite{windtvdl:04,Marty200958,Samper2010278};
\item bio-geochemistry involves interactions with organic chemicals, and
  is important for studies of soil pollution, and also for
  implementing bio-remediation policies~\cite{WRCR:WRCR7807,MARZAL1994363}.
\end{itemize}

\medskip
The numerical simulation of reactive transport has been the topic of
numerous work. The survey by Yeh and Tripathi~\cite{yehtrip} has been
very influential in establishing a mathematical formalism for setting
up models, and also for establishing the ``operator splitting''
approach (see below) as a standard. More recent surveys, detailing
several widely used computer codes and their applications, can be found
in the book~\cite{zhangyehparker:reactransbook} and the \MK{survey}
article~\cite{Steefel2015}. 

\MK{One traditionally distinguishes} two main approaches for solving
reactive transport problems:
\begin{description}
\item[Sequential Iterative Approach (SIA)] \MK{in this family of
    methods (also known as  operator splitting), transport and chemistry are solved
  alternatively}~\cite{carrmosbeh:04,lagneau:10,WRCR:WRCR9205,Samper2008,WRCR:WRCR5571}.
  This has the advantage that a code for reactive transport can be
  built from pre-existing transport and chemistry
  codes~\cite{jaradreuzycoche:compgeo17}, and that no global system of
  equations needs to be solved. On the other hand, the splitting between
  chemistry and transport may restrict the time step in order to
  ensure convergence of the method, and the splitting may also
  introduce mass errors that need to be
  controlled~\cite{WRCR:WRCR5757}. As the above references show, the
  method can be quite successful if it is implemented carefully.
\item[Globally implicit Approach (GIA)] \MK{Here one solves the fully
    coupled system involving transport and chemistry is solved in one
    shot (see for instance~\cite{Fahs2008,sacaay:01} and the
    references below).} The balance is the opposite of what it was for SIA: the
    method does not introduce spurious mass errors, and can converge
    with large time steps, but it leads to a large and difficult to
    solve system of non-linear equations coupling all the chemical
    species at all grid points. \MK{The system can be solved by
    substituting the mass action laws into the conservation equations,
    a variant known as Direct Substitution approach (DSA) as
    in~\cite{fan2012fully,hammvallicht05}. Recently, methods that
    solve the coupled problem without direct substitution have been
    introduced. In~\cite{Hoffmann2009,knabner05,knabner07},
    local chemical systems are solved, and the solution becomes an additional term in a transformed form of the
    transport. In~\cite{deDieuleveult2010,Erhel:2017:AGR:3067934.3068012,ehrsabdieul:13},
    the overall coupled system is solved as a Differential Algebraic
    System. 
}
  \end{description}


In addition to the surveys mentioned above, the methods have been
compared in several studies or
benchmarks~\cite{Momascompar,Marty2015,sacaay:00}. It is fair to say
that no method or code emerges as a clear winner for all situations. 

We finish this short (and far from exhaustive) review of the
literature by noting that most of the work cited deal with one-phase
flow and transport. The methods have recently been extended to the
case of two-phase
flow~\cite{refId0,fan2012fully,12985679,SIN201762} . 

\medskip \MK{In a previous paper~\cite{compgeo:10b}, the authors
  introduced a method that belongs to the GIA family without
  sacrificing the ease of implementation of the SIA methods, due to the
  separation of software modules for chemistry
  and transport. The method solves the nonlinear coupled system by a
  Jacobian-free Newton--Krylov method. 

This paper presents several improvements to~\cite{compgeo:10b}:
 \begin{itemize}
 \item a systematic study of preconditioning methods for the
   \emph{linearized coupled problem}, and their relationship to elimination
   methods. This results in a method where the mobile concentrations
   are eliminated, and one solves for the total fixed
   concentrations. We give both heuristic arguments and experimental
   evidence that the convergence rate for both the Newton and GMRES
   iterations is bounded independently of the mesh size;
\item a formulation of the \emph{chemical equilibrium problem} that does not
  need the a priori separation of the chemical species between primary
  and secondary components, but that keeps the distinction between
  mobile and immobile species.
 \end{itemize}

 The method originally introduced in~\cite{compgeo:10b} approximated
 the Jacobian matrix by vector product by finite difference, so that
 only solvers for transport and chemistry were needed, and they could
 be applied as black boxes. In this work, we slightly ``open the
 boxes'', as we propose to compute exactly the Jacobian matrix by vector
 products, so that access to the Jacobians of the transport and
 chemical solvers are needed. We feel the added accuracy and reduced
 cost (for the chemical part, the inverting the jacobian is much less
 expensive than solving the whole system) more than make up for the
 additional requirement on the software.

 Since the convergence of Krylov solvers can be slow, we devote a
 significant part of the paper to the analysis of preconditioning
 methods. We show that block preconditioning is equivalent to the
 elimination of variables, and that the elimination can be performed
 directly on the non-linear system.

In a different direction, this paper also relaxes the reliance on the
Morel formulation that requires that one identifies a priori a set of
principal and secondary species. We show that the coupled problem can
be formulated in the same way as before, but that the various totals
involved can be defined in a more intrinsic way. This may be seen as a
particular case of the general reduction mechanism of Knabner et
al~\cite{Hoffmann2009,knabner07}, but we believe the simpler
formulation may be of interest to practitioners. The formulation also
leads to a method for solving the chemical system that makes uses of
orthogonal matrices and may lead to better conditioned Jacobians (this
remains to be investigated). 

In this work, we consider only a
simplified physical and chemical setup (one phase flow, no mineral
reactions, no kinetic reactions) so as to concentrate on the numerical
issues related to the coupling between transport and
chemistry. Generalization to more realistic situations (2D and 3D
  geometries, kinetic reactions, presence of mineral species) will be the
  topic of future studies.
}



\medskip
An outline for the rest of this paper is as follows: we set
up the mathematical model in section~\ref{sec:rtm}, and show how to
reduce the problem by eliminating the chemical
concentrations. Numerical methods for solving the local chemical
equilibrium problem and the advection--diffusion equations for
transport are the topic of sections~\ref{sec:chemeq}
and~\ref{sec:convnum} respectively. Section~\ref{sec:coupled} deals
with the formulation of the coupled problem, and its solution by a
Newton--Krylov method. Preconditioners for the linear system and
elimination methods for the non-linear system are detailed in
section~\ref{sec:precond}.  Finally, in section~\ref{sec:numres} the
methods are validated \MK{on two test cases, including} the 1D MoMaS
reactive transport benchmark, which is a fairly difficult test case
for reactive transport~\cite{Momascompar}.

\section{Reactive transport model}
\label{sec:rtm}

We consider a set of species subject to transport by advection and
diffusion and to chemical reactions in a porous medium \MK{occupying a
  domain $\Omega \subset \Rn[d]$ ($d=1,2,3$)}. \MK{[-]}
The chemical phenomena involve both homogeneous and heterogeneous
reactions. Homogeneous reactions, in the aqueous phase, include water
dissociation, acid--base reactions and \MK{redox} reactions, whereas
heterogeneous reactions occur between the aqueous and solid phases, and
include surface complexation, ion exchange and precipitation and
dissolution of minerals (see~\cite{AppelPost:05} for details on the
modeling of specific chemical phenomena). Accordingly, we assume there
are $N_s$ mobile species $(X_j)_{j=1, \ldots, N_S}$ in the aqueous
phase and $\bar{N}_s$ immobile species in the solid phase
$(\bar{X}_j)_{j=1, \ldots, \bar{N}_S}$, and that there are $N_r$
homogeneous reactions, and $\bar{N}_r$ heterogeneous reactions.

In this work, we only consider equilibrium reactions, which means that
the chemical phenomena occur on a much faster scale than transport
phenomena, see~\cite{rubin:83}. This assumption is justified for
aqueous phase and ion--exchange reactions, but may not hold for
reactions involving minerals. Such reactions should be modeled as
kinetic reactions, with specific rate
laws~\cite{bear:2010,fan2012fully}. 

We can write the chemical system as 
\begin{equation*}
    \sum_{j=1}^{N_s} (S_{cc})_{ij} X_j \leftrightarrows 0 \quad
    i=1, \ldots, N_r \quad \text{ homogeneous reactions}, 
  \end{equation*}
\vspace{-2em}
  \begin{equation*}
  \sum_{j=1}^{N_s}(S_{\bar{c}c})_{ij} X_j  + \sum _{j=1}^{\bar{N}_s} (S_{\bar{c}\bar{c}})_{ij} \bar{X}_j \leftrightarrows 0
     \quad i=1, \ldots \bar{N}_r \quad \text{heterogeneous reactions},
\end{equation*}
or in condensed form
\begin{equation}
  \label{eq:chemequmat}
\MK{S    \begin{pmatrix}    X \\ \bar{X}   \end{pmatrix}
= }
\begin{pmatrix}
   S_{cc} & 0 \\ S_{\bar{c}c} & S_{\bar{c}\bar{c}} 
  \end{pmatrix}
  \begin{pmatrix}    X \\ \bar{X}   \end{pmatrix}
\leftrightarrows
  \begin{pmatrix}   0 \\ 0   \end{pmatrix}.
\end{equation}
We \MK{denote by $S$} the stoichiometric matrix, 
with the sub-matrices $S_{cc} \in \Rn[N_r \times
N_s]$, $S_{\bar{c}c} \in \Rn[\bar{N}_r \times N_s]$ and
$S_{\bar{c}\bar{c}} \in \Rn[\bar{N}_r \times \bar{N}_s]$. 
We assume that both the global stoichiometric matrix $S$ and the
``aqueous'' stoichiometric matrix $S_{cc}$ are of full rank. As there
will usually be more species than reactions, this just
means that all reactions are ``independent'' (in the linear algebra
sense !), and that this is also true of the reactions in the aqueous
phase.  

Each reaction gives rise to a mass action law, linking the activities
of the species. For simplicity, we assume all species follow an ideal
model, so that their activity is equal to their concentration. We
denote by $c_j$ (resp. $\bar{c}_j$) the concentration of species $X_j$
(resp. $\bar{X}_j$). It will be convenient to write the mass action
law in logarithmic form, so for a vector $c$ with positive entries, we
denote by $\log c$ the vector with entries $\log c_j$. We then have
\begin{equation}
  \label{eq:massaction}
\begin{pmatrix}
    S_{cc} & 0 \\ S_{\bar{c}c} & S_{\bar{c}\bar{c}} 
  \end{pmatrix}
  \begin{pmatrix}    \log c \\ \log \bar{c}   \end{pmatrix}
=
  \begin{pmatrix}   \log K \\ \log \bar{K}   \end{pmatrix}, 
\end{equation}
where $K \in \Rn[N_r]$ and $\bar{K} \in \Rn[\bar{N}_r]$ are the
equilibrium constants for their respective reactions. 

\medskip
Next, we write the mass conservation for each species, considering both
transport by advection and diffusion and chemical reaction terms:
\begin{equation}
  \label{eq:massconv}
  \begin{aligned}
    & \phi \partial_t c & + \Lop c &= S_{cc}^T r + S_{\bar{c}c}^T \bar{r}, \\
    & \phi \partial_t \bar{c} & &= S_{\bar{c}\bar{c}}^T \bar{r},
  \end{aligned} \quad \text{in } \Omega \times [0, T_f]
\end{equation}
\MK{where $\Lop$ denotes the advection--diffusion operator (written in
  1D):
 \begin{equation*}
\Lop(c) = \partial_x 
 \left( -D \partial_x c +u c \right),
 \end{equation*}
$\phi$ is the porosity (fraction of void in a Representative
Elementary Volume available for the flow), $u$ is the Darcy velocity
(we assume here permanent flow, so that $u$ is considered as known),
$D$ is a diffusion--dispersion coefficient and $r \in \Rn[N_r]$ and $\bar{r} \in
\Rn[\bar{N}_r]$ are vectors containing the reaction rates. 
We assume that
the diffusion coefficient is independent of the species. This is a
strong restriction on the model, but one that is commonly assumed to
hold~\cite{compgeo:10b,knabner05,saalayorcarr98,yehtrip}. The model is
completed by appropriate initial and boundary conditions.}

Since we
assume that all reactions are at equilibrium, \MK{the reaction rates} are unknown
and we now show how they can be eliminated.

\subsection{Elimination of the reaction rates -- the coupled problem}
We follow the approach of Saaltink et al.\cite{saalayorcarr98} (see
also a more general approach in~\cite{knabner05,knabner07}) by
introducing a \emph{kernel matrix} $U$ such that $ U S^T =0$, i.e.\
such that  columns of $U^T$ form of basis for the null-space of
$S$. This can be done in several ways (see the above references, and
section~\ref{sec:chemeq}). Here we outline how one can compute such a
matrix with the same structure as that of $S$. 

\begin{lem}
\label{lem:kern}
Assume that the matrix $S$ is as defined in equation~\eqref{eq:chemequmat},
that it has full rank, and that the submatrix $S_{cc}$ also has full
rank. Then there exists a \MK{kernel} matrix 
\begin{equation}
  \label{eq:Umat}
    U = \begin{pmatrix}
    U_{cc} & U_{c\bar{c}} \\ 0 & U_{\bar{c}\bar{c}} 
  \end{pmatrix}, 
\quad 
\begin{aligned}
&U_{cc} \in \Rn[(\MK{N_s-N_r}) \times N_s], \; U_{c\bar{c}} \in\Rn[(\MK{N_s-N_r}) \times \bar{N}_s] \\
&  U_{\bar{c}\bar{c}} \in \Rn[(\MK{\bar{N}_s - \bar{N}_r}) \times
\bar{N}_s]
\end{aligned}
\end{equation}
such that $U S^T = 0$.
\end{lem}
\begin{proof}

The existence of a kernel matrix $U$ is well known (see any standard
text on linear algebra, such as~\cite{strang:introLA09}
or~\cite{meyer:bk00}). The content of the lemma is that the kernel
matrix can be chosen with a block upper triangular structure. 

The proof proceeds by computing each block of the product, and showing
that the blocks in $U$ can be chosen as specified. This is true
because both $C_{cc}$ and $C_{\bar{c}\bar{c}}$ are of full rank (the
later because the full matrix $S$ is block triangular).
%
\end{proof}

With the kernel matrix constructed in lemma~\ref{lem:kern}, we can
eliminate the reaction terms in equation~\eqref{eq:massconv}. 
We start by defining the total analytic concentration for the
mobile and immobile species respectively (these are the same as
various total quantities defined in the classical survey by Yeh and
Tripathi~\cite{yehtrip}) 
\begin{equation}
  \label{eq:deftotan}
  \begin{pmatrix}    T \\ \bar{T}   \end{pmatrix} = 
  \begin{pmatrix}
    U_{cc} & U_{c\bar{c}} \\ 0 & U_{\bar{c}\bar{c}} 
  \end{pmatrix}
  \begin{pmatrix}
    c \\ \bar{c}
  \end{pmatrix}.
\end{equation}
We also define the total mobile and immobile
concentrations for the species in the aqueous phase
\begin{equation}
\label{eq:deftot}
    C =  U_{cc}\, c, \quad  \overline{C} =    U_{c\bar{c}}\, \bar{c}, 
\end{equation}
so that the total concentrations are given by
\begin{equation}
\label{eq:deftotal}
  T = C+ \bar{C} = U_{cc}\, c + U_{c\bar{c}}\, \bar{c}.
\end{equation}
These new unknowns are all of dimension $N_c$, where $N_c = N_s-N_r$. 
\MK{We now multiply system~\eqref{eq:massconv} on the left by $U$.}
Because of our assumption that $D$ is the same for all chemical
species, multiplication by $U$ commutes with  the differential
operator, and the system can be rewritten as
\begin{equation}
\label{eq:coupled}
  \begin{aligned}
    &\phi \partial_t C + \phi \partial_t \overline{C} + \Lop C &= 0, \\
    & \phi \partial_t \bar{T} &= 0.
  \end{aligned}
\end{equation}

The coupled system consists of the $(N_s-N_r)+(\bar{N}_s-\bar{N}_r)$
conservation PDEs and ODEs~\eqref{eq:coupled}, together with the $N_r +
\bar{N}_r$ mass action laws~\eqref{eq:massaction} and the relations
connecting concentration and totals~\eqref{eq:deftot} \MK{and the
  second line of~\eqref{eq:deftotan}}, for the $N_s+
\bar{N}_s$ concentrations and $2(N_s-N_r) + \bar{N}_s-\bar{N}_r$
totals. Note that the ODEs for $\bar{T}$ are decoupled from the rest
of the system.  In section~\ref{sec:coupled}, we show how to eliminate
the individual concentrations from the system, so that only the totals
$C$ and $\bar{C}$ remain as unknowns.

\begin{remark}
  The formulation for the chemical system given in
  equations~\eqref{eq:massaction} and~\eqref{eq:deftotan} generalizes
  the well known Morel formulation~\cite{Morel}, where a set of
  ``principal'' species is identified, and the remaining ``secondary''
  species are written in terms of the principal ones.

The stoichiometric matrix $S$ is split naturally in blocks, the mass
action laws allow the elimination of the secondary unknowns, and the
conservation laws lead to a non-linear system of equations for the
principal species. 

In the next section, we show that the ``local chemical system''
\MK{~\eqref{eq:massaction},~\eqref{eq:deftotan} can be solved (for given
$(T, \bar{T})$)} in a
similar way, without having first to explicitly identify the
principal and secondary species. 
\end{remark}

\section{Methods for solving chemistry and transport}

We briefly recall in this section how the chemical equilibrium system,
and the transport equation are solved. 

\subsection{The chemical equilibrium problem}
\label{sec:chemeq}

The subsystem formed by the mass action laws~\eqref{eq:massaction} and
the definition of the totals~\eqref{eq:deftotan} is a closed system
that enables computation of the individual concentrations $c$ and
$\bar{c}$ given the totals $T$ and $\bar{T}$. This is actually the same
system that would be obtained in a closed chemical system where now
$T$ and $\bar{T}$ would be known as the input total concentrations. This
subsystem is what we call in the following the \emph{chemical
  equilibrium problem}. It is a small (its size is the number of
species) nonlinear system, that is notoriously difficult to solve
numerically (see for
example~\cite{Carrayrou2002,Dieuleveult2009,Hoffmann2009,AIC:AIC15506}).

It will be convenient in this subsection to temporarily ignore the
distinction between mobile and immobile species. We thus define the
vectors 
\begin{equation*}
\xi = \begin{pmatrix} c \\ \bar{c} \end{pmatrix} \in \Rn[N_{\xi}], 
\quad   
\kappa = \begin{pmatrix} \log K \\ \log
  \bar{K} \end{pmatrix} \in \Rn[N_{\kappa}]
\quad 
\tau
= \begin{pmatrix} T \\ \bar{T} \end{pmatrix} \in
\Rn[(N_{\xi}-N_{\kappa})]
\end{equation*}
($N_{\xi} = N_s + \bar{N}_s$ is the total number of species and
$N_{\kappa} = N_r + \bar{N}_r$ is the total number of reactions), and
write the problem for the chemical equilibrium as
\begin{equation}
  \label{eq:chemeq}
  \begin{aligned}
    & S \log \xi = \kappa, \\
    & U \xi = \tau,
  \end{aligned}
\end{equation}
with the stoichiometric matrix $S \in \Rn[N_{\kappa} \times N_{\xi}]$,
and the kernel matrix $U \in \Rn[(N_{\kappa}-N_{\xi}) \times N_{\xi}]$ is such that $U
S^T=0$.

It has been shown by several
authors~\cite{compgeo:10b,Carrayrou2002,Hoffmann2009} that taking the
logarithms of the concentrations as new unknowns in~\eqref{eq:chemeq}
is beneficial from the numerical point of view, as it automatically
ensures that all concentrations will be positive (which might otherwise
be difficult to enforce), and it makes all the unknowns to be of
comparable size (in some cases, such as \MK{redox} reactions,
the concentrations have been seen to vary over several orders of
magnitude). We take the same convention to define the exponential of a
vector as for the logarithm: for a vector $z$ the
vector $\exp(z)$ is defined so that its $i$th entry is
$\exp(z_i)$. By defining $z = \log(\xi) \in \Rn[N_{\xi}]$,
system~\eqref{eq:chemeq} thus becomes
\begin{equation}
  \label{eq:chemeqlog}
  \begin{aligned}
    & S z = \kappa, \\
    & U \exp(z) = \tau,
  \end{aligned}
\end{equation}

In order to solve system~\eqref{eq:chemeqlog}, we take a cue from the
solution of constrained least squares problem
(see~\cite[chap. 5]{bjor:97}, and also~\cite{erhel:hal-01584490} in
the same context): we consider the first equation as a constraint, and
determine its general solution as the sum of a particular solution,
and an unknown of smaller dimension, that will be found by
substituting it into the second equation. Following the references above, we
first compute a $QR$
decomposition of $S^T$, as $Q^T S^T = \begin{pmatrix}  R \\
  0 \end{pmatrix}$,
with $Q= (Q_1, Q_2) \in \Rn[N_{\xi} \times N_{\xi}]$ orthogonal and
$R \in \Rn[N_{\kappa} \times N_{\kappa}]$ upper triangular (and
invertible, as $S$ is assumed to be of full rank).

The columns of $Q_1 \in \Rn[N_{\xi} \times N_{\kappa}]$ form an
orthonormal basis of $\text{range } S^T$, while those of $Q_2 \in
\Rn[N_{\xi} \times (N_{\xi}-N_{\kappa})]$ form an orthonormal basis of
$\text{Ker } S$. A first consequence is that one may choose $U=Q_2^T$
in the second equation of~\eqref{eq:chemeqlog}. 

Now, any $z \in \Rn[N_{\xi}]$ may be written uniquely as $z= Q_1 y_1 + Q_2
y_2$, with $y_1 \in \Rn[N_{\kappa}]$ and $y_2 \in
\Rn[N_{\xi}-N_{\kappa}]$. If we substitute this expression in the
first equation of~\eqref{eq:chemeqlog}, we obtain 
\begin{equation*}
  R^T y_1 = \kappa,
\end{equation*}
which determines $y_1$ uniquely. We set $b_1 = Q_1 R^{-T} \kappa$, so
that $z= b_1 + Q_2 y_2$ ,with $b_1$ known. 

We now use the expression for $z$ in the second equation
of~\eqref{eq:chemeqlog}, and we obtain a non-linear system of size
$N_{\xi} - N_{\kappa}$ for $y_2 \in \Rn[N_{\xi} - N_{\kappa}]$:
\begin{equation}
\label{eq:chemreduced}
H(y_2) \underset{\text{def}}{=} Q_2^T \exp(b_1 + Q_2 y_2) - \tau = 0. 
\end{equation}
This is the system that is to be solved numerically, and that is
analogous to the non-linear system for the principal species obtained
via the Morel formalism. The number of unknowns has been reduced from
the number of species to the number of species minus the number of
reactions, that is the number of principal species. 

As a side remark, let us notice that the Jacobian matrix for this
reduced chemical system has the form
\begin{equation*}
  J_c = Q_2^T \diag(c)\; Q_2, \quad c= b_1 + Q_2 y_2,
\end{equation*}
and this form is again the same as that obtained from the Morel
formulation, with the difference that here matrix $Q_2$ is
orthogonal. This may be important as the Jacobian matrices occuring in
the solution of the chemical equilibrium problem may be very ill
conditioned, as shown in~\cite{Machat2017}. Thus using orthogonal
matrices limits the ill-conditioning to the diagonal matrix $\diag(c)$
where it is harmless. 

To solve the reduced chemical problem~\eqref{eq:chemreduced}, we used a variant of Newton's
method. As is well known Newton's method does not always converge in practice and, especially
for a code that is designed to be used in coupling applications, it is
essential to ensure that the solver ``always'' works. We have found
that using a globalized version of Newton's method (using a line
search, cf.~\cite{Kelley95}) was effective in making the algorithm
converge from an arbitrary initial guess.

\medskip We now return to the context of solving the coupled problem,
where it is important to distinguish between mobile and immobile
species. Indeed, what is needed is the partition of the species
between their mobile and immobile forms, rather than the individual
concentrations (though they are still needed as intermediate
quantities). The totals can be computed a posteriori using their
definitions in equation~\eqref{eq:deftot}. 

It will be convenient to condense the chemical sub-problem by a
function
\begin{equation}
\label{eq:defpsic}
  \begin{aligned}
    \psi_C : \quad & \mathbf{R}^{Nc} \to \mathbf{R}^{Nc} \\
           & T \mapsto \psi_C(T) =\bar{C} =  U_{c\bar{c}}\, \bar{c}
  \end{aligned}
\end{equation}
where $\bar{c}$ is obtained by solving the chemical
problem~\eqref{eq:chemeqlog}, given $T$ (and $\bar{T}$, which we take as a
constant) for $\xi$, and computing $\bar{C}$ as indicated above.

\subsection{Transport model} 
\label{sec:convnum}
\MK{In this section we denote by $c$ a generic unknown concentration. When
the transport system is solved in the context of a coupled problem,
the role of $c$ will be played by $C$ as defined in
Section~\ref{sec:rtm}, cf. equation~\eqref{eq:coupled}.}
The transport of a single species through a 1D porous medium
$\Omega=]0,L[$ \MK{is governed by the advection--diffusion equation
  (cf.~\eqref{eq:massconv}):}
\MK{
\begin{equation}
\begin{array}{l}
\phi \partial_t c +\partial_x
 \left( -D \partial_x c +u c \right)= \phi \partial_t c + \partial_x j = \phi q, \quad 0 <
  x < L, \quad  0 < t < T_f,
\end{array}
\label{eq1}
\end{equation}
implicitly defining the flux $j$.}

The initial condition is $c(x,0)=c_0(x)$ and, in view of the
applications, the boundary conditions are a Dirichlet condition (given concentration)
$c(0,t)=c_d(t)$ at the left boundary  ($x=0$) and zero diffusive flux
$\dfrac{\partial c}{\partial x} = 0$ at the right
boundary ($x=L$). We could easily take into account more general
boundary conditions. 

\subsubsection{Discretization in space}
We treat the space and time discretization separately, as we will use
different time discretizations for the different parts of the
transport operator. 

For space discretization we use a cell-centered finite volume
scheme~\cite{eymagallherb:00}. 
The interval $]0,L[$ is divided into $N_h$ intervals
$[x_{i-\frac{1}{2}},x_{i+\frac{1}{2}}]$ of length $h_i$, where
$x_{\frac{1}{2}}=0,x_{N_h+\frac{1}{2}}=L$.
For $i=1,..,N_h$ we denote by $x_i$ the center and $x_{i+1/2}$ the extremity
of the element $i$.   We denote by $c_i$, $i=1,..,N_h$ the approximate concentration in cell $i$. 

\MK{We split the flux in equation~\eqref{eq1} 
as the sum of a diffusive flux $j_d=-D\dfrac{\partial
  c}{\partial x}$ and an advective flux $j_a=u c$.}
We then integrate equation~\eqref{eq1} over a cell $]x_{i-1/2},
x_{i+1/2}[$, to obtain
\begin{equation}
\label{eq:trdisc}
\phi_i h_i \dfrac{d c_i}{dt} + j_{d,i+\frac{1}{2}} +
j_{a,i+\frac{1}{2}} - j_{d,i-\frac{1}{2}} -
j_{a,i-\frac{1}{2}} = h_i \phi_i q_i, \qquad i=2, \dots, N_h. 
\end{equation}
Our flux approximations come from finite differences. For the diffusive
flux, we use harmonic averages for the diffusion coefficient (as used
in mixed finite element methods) :
\begin{equation}
\label{eq:fluxdisc}
j_{d,i+\frac{1}{2}}=-D_{i+\frac{1}{2}}\left(\frac{c_{i+1} -
    c_i}{h_{i+\frac{1}{2}}}\right)   
\end{equation}
with
$$
D_{i+\frac{1}{2}}=\frac{2D_iD_{i+1}}{D_i+D_{i+1}}, \quad
D_{\frac{1}{2}}=D_1 \quad D_{N_h+\frac{1}{2}}=D_{N_h}
 \quad \mbox{and}\quad  h_{i+\frac{1}{2}}=\frac{h_i+h_{i+1}}{2}
$$ 
For the advective flux, we use an upwind approximation, so that
(assuming for simplicity that $u>0$), $j_{a,i+\frac{1}{2}}=u c_i$

These approximations are corrected to take into account the boundary
conditions, both at $x=0$ and at $x=L$. We give a matrix formulation,
keeping time continuous for now. Since we will be using different
discretizations for the diffusive and advective parts (see next
section), we keep the matrices for advection and diffusion separate. 
With the notation:
\begin{equation*}
  \alpha_i = \dfrac{D_{i+\frac{1}{2}}}{h_{i+\frac{1}{2}}} +
  \dfrac{D_{i-\frac{1}{2}}}{h_{i-\frac{1}{2}}}, \; \beta_i =
\MK{-}  \dfrac{D_{i-\frac{1}{2}}}{h_{i-\frac{1}{2}}}, \; \gamma_i =
 \MK{-} \dfrac{D_{i+\frac{1}{2}}}{h_{i+\frac{1}{2}}}, \quad i=2,
  \ldots, N_h-1,
\end{equation*}
we define the matrices (with appropriate modifications for the
boundary terms)
\begin{equation*}
  A_d = \text{tridiag}(\beta_i, \alpha_i, \gamma_i), \quad A_a =
  \text{tridiag}(-u, u, 0)
\end{equation*}
as well as $M=\diag(\phi_i h_i)$.

The semi-discrete system can be written as
\begin{equation}
  \label{eq:semidisc}
  M \dfrac{dc}{dt} +(A_d + A_a) c = M q,
\end{equation}
with the initial condition $c_i = \frac{1}{h_i}\int_{x_{1-1/2}}^{x_{i+1/2}} c_0(x_i)$.

\subsubsection{Time discretization}

As the transport operator contains both advective and diffusive terms,
it makes sense to use different time discretization methods for the
different terms. Specifically, the diffusive terms should be treated
implicitly, and the advective terms are better handled
explicitly. Similarly to what was done above, we discretize the
interval $[0, T_f]$ with a time step $\Delta t$, which we take as a constant
for simplicity, and we denote by $c_i^n$ the (approximate) value of
$c_i(n \Delta t)$, and by $c^n$ the corresponding vector.

We compare several methods for discretization in time.

  \begin{description}
\item[Fully implicit]  both the diffusive and advective terms are
  treated implicitly: at each time step, we solve a linear system 
  \begin{equation*}
    (M+\Delta t (A_d + A_a)) c^{n+1} = M c^n + M \Delta t q^{\MK{n+1}}.    
  \end{equation*}
\item [Explicit advection and implicit diffusion] the diffusive terms are treated implicitly, 
and advective terms are handled explicitly under a CFL condition.
With this scheme, the system to be solved at each time step is
\begin{equation*}
  (M+ \Delta t A_d) c^{n+1} = (M - \Delta t A_a ) c^n + M \Delta t q^{\MK{n+1}}.
\end{equation*}
As it has an explicit part, the scheme just defined is stable under
the CFL condition $\phi_i u \Delta t \le \max_i h_i$. 

As this may be too
severe a restriction (some of our applications require integration
over a very large time interval), we use an operator splitting
scheme proposed by Siegel et al.~\cite{siegmosackjaff:97} (see
also~\cite{hotackmos:04}) that is both unconditionally stable, and
has a good behavior in advection dominated situations.
\item [Splitting] (Explicit advection and implicit diffusion), with
    sub-time steps : this scheme works by taking several small time
    steps of advection, controlled by 
CFL condition within a large time step of diffusion. Thus CFL impacts
advection, and larger time steps can be taken for diffusion.
 
More precisely, the time step $\Delta t$ will be used as the diffusion
time step, it is divided into $N_a$ time steps of advection
$\Delta t_a$ such that $\Delta t= N_a \Delta t_a$ where $N_a \geq 1$,
the advection time step will be controlled by CFL condition. Note that
taking a single advection step amounts to using the implicit--explicit
method seen previously. We solve equation~\eqref{eq1} over the time
step $[t^n, t^{n+1}]$ by first solving the advection equation
$\phi \partial_t c +\partial_x (u c)=0$ over $N_a$ steps of size
$\Delta t_a$ each, and then solve the diffusion equation
$\phi \partial_t c + \partial_x (-D \partial_x c)=\phi q$ starting
from the value at the end of the advection step.

\paragraph{\textbf{Advection step}}

Denote the intermediate times by $t^{n,m}, m=0, \dotsc, N_a$, with
$t^{n,0}=t^{n},\,t^{n,N_a}=t^{n+1}$. Each interval $[t^{n},t^{n+1}]$
is then divided into $N_a$ intervals 
$[t^{n,m},t^{n,m+1}], m=0,...N_a-1$.  Let $c^{n,m}$ be the
approximate concentration $c$ at time $t^{n,m}$ and $c^{n,0}=c^{n}$. We
discretize the advection equation in time, using the explicit Euler
method, we obtain
\begin{equation}
\label{eq:advsplit}
  M c^{n, m+1} = (M \MK{-} \Delta t_a A_a) c^{n,m}, \quad m=0, \dots, N_a-1.
\end{equation}

\paragraph{\textbf{Diffusion step}}
The diffusion part is discretized by an implicit Euler scheme,
starting from $c^{n, N_a}$ :
\begin{equation}
  (M + \Delta t A_d) c^{n+1} = M c^{n, N_a} + M \Delta t q^{\MK{n+1}}.
\end{equation}

\end{description}

For further use, we note that all three discretizations methods can be
written in a similar way, as
\begin{equation}
  \label{eq:transdisc}
  A c^{n+1} = B c^n + \MK{M} \Delta t q^{\MK{n+1}},
\end{equation}
where the matrices $A$ and $b$ are defined in each case as: 
  \begin{description}
  \item[Fully implicit] $A = M + \Delta t (A_d+A_a)$ and  $B=M$, 
  \item[Explicit--implicit]  $A = M + \Delta t A_d$ and  $B =M -\Delta
    t A_a $,
  \item[Splitting] $A = M + \Delta t A_d$ and $B$ is defined implicitly
  by~\eqref{eq:advsplit}
  \end{description}

Alternatively, we may want to follow the pattern begun in
section~\ref{sec:chemeq}, and abstract one transport step as a mapping
from $c^n$ to $c^{n+1}$ under the action of the source $q^{\MK{n+1}}$. We
define
\begin{equation}
  \label{eq:tranpsol}
 \begin{aligned}
\psi_T:  \quad
   & \Rn[N_h] \times \Rn[N_h]  \to \Rn[N_h] \\
   & (c,q) \mapsto \psi_T(c,q ) = A^{-1} (B c + M \Delta t q)
  \end{aligned}
\end{equation}

\section{Methods for solving the coupled system}
\label{sec:coupled}

In this section, we discuss methods for solving the coupled
transport--chemistry system. It may be difficult to compute
and store the Jacobian matrix (it may be very large, and contains a
block that is the inverse of the chemical solution operator, that may
be difficult to compute). Consequently we advocate a Newton--Krylov approach, as
this requires only the multiplication of the Jacobian matrix by a
given vector, the Jacobian matrix itself does not have to be
stored. By exploiting the block structure of the Jacobian matrix, the
matrix--vector product can be computed in terms of the individual
Jacobians.

The efficiency of the Newton--Krylov method rests on the choice of an
adequate preconditioner (see~\cite{knollkeyes04} for several
illustrations). In this paper, we show that block preconditioners for
the Jacobian can be related to physics--based approximations, and also
that a block Gauss--Seidel preconditioner is closely related to an
elimination method at the non-linear level.

\subsection{Formulation of the coupled system}
\label{sec:formul}

We start with the coupled system that was obtained at the end of
section~\ref{sec:rtm}.  It consists of the (transformed) transport
equations~\eqref{eq:coupled}, together with the mass action
laws~\eqref{eq:massaction}. They are linked by the definition of the
transformed variables $T, C$ and $\bar{C}$ in
equations~\eqref{eq:deftot} and \eqref{eq:deftotal}.

Because chemistry is local, we can eliminate the individual
concentrations at each point by using the ``chemical
solution'' operator $\psi_C$ defined in
equation~\eqref{eq:defpsic}. This only leaves $C$, $\bar{C}$, $T$ and
$\bar{T}$ as unknowns that are solution of the following system
\begin{equation}
  \label{eq:coupledbis}
  \begin{aligned}
  & \phi \partial_t C + \phi \partial_t \bar{C} + \Lop C &=&\; 0, \\
  & \phi \partial_t \bar{T} &=& \; 0, \\
  & T &=&\; C +\bar{C}, \\
  & \bar{C} &=& \;\psi_C(T).
  \end{aligned}
\end{equation}

As is clear from the second equation, $\bar{T}$ is constant in time
(we will see that it is simply the cationic exchange capacity in the
example below). To simplify the notation, the equation
for $\bar{T}$ will be omitted in the sequel, and it has been dropped
from the definition of $\psi_C$.

The next step is to discretize the transport equations in space and
time, using any of the methods that were discussed in
section~\ref{sec:convnum}. For ease of notation, we denote the
discrete unknowns (vectors of size $N_h$) by the same letter in bold
typeface as their continuous counterpart. Moreover, all quantities at
time $t^n$ will be indicated by a superscript $n$.

We make use of a notational device, inspired from Matlab, that was
introduced in~\cite{compgeo:10b} to take into account the dependence
of the unknowns both on the grid cell index and the chemical species
index. For a block vector $u_{ij}$, where $i \in [1, N_c]$ represents the
chemical index and $j \in [1, N_h]$ represents the spatial index, we
 denote by
\begin{itemize}
\item $u_{:,j}$ the column vector of concentrations of all chemical species in grid cell $j$.  
\item $u_{i,:}$ the row vector of concentrations of species $i$ at all
  grid points;
\end{itemize}
The unknowns are thus numbered first by chemical species, then by grid
points, so that all the unknowns for a single grid point are numbered
contiguously. As proposed by Erhel et al.~\cite{ehrsabdieul:13}, we
make use of the Kronecker product and the $\vecop$ notation (see for instance ~\cite{golubvanloan4:13}).
Given $V = \begin{pmatrix}  V_1, V_2, \dotsc, V_{N_c}   \end{pmatrix}
\in \Rn[N_x \times N_c]$ ($V_i$ denotes the $i$th column of $V$) we denote by $\vecop(V)$ the vector in
$\Rn[N_x Nc]$ such that
\begin{equation*}
  \vecop(V) =  \begin{pmatrix}  V_1 \\ V_2 \\ \vdots, \\ V_{N_c}   \end{pmatrix}.
\end{equation*}

We also
extend the solution operator $\psi_C$ (defined in \eqref{eq:defpsic}
to operate on the global vector:
\begin{equation}
  \begin{aligned} \Psi_C: \quad & \Rn[N_c \times N_h] \mapsto
\Rn[N_c \times N_h] \\ &\bT \to \Psi_C(\bT), \quad \text{with }
\Psi_C(\bT)_{:,j} &= \psi_C(\bT_{:,j}), \; \text{for } j =1, \dotsc, N_h
  \end{aligned}
\end{equation}
and define the right hand side vector by $\mathbf{b}^n = B \bC^n + M \bCbar^n$.

With these conventions, the discretized version of the coupled system~\eqref{eq:coupledbis} becomes
\begin{equation}
  \label{eq:coupleddisc}
  \begin{aligned}
    & A {(\bC_{i,:}^{n+1})}^T + M {(\mathbf{\bar{C}}_{i,:}^{n+1})}^T -{(\mathbf{b}_{i,:}^n)}^T = 0, & i=1, \dotsc, N_c, &\\
  &\bT_{ij}^{n+1} - \bC_{ij}^{n+1} - \bCbar_{ij}^{n+1} = 0, & i=1, \dotsc, N_c, \;
     & j=1, \dots, N_h,\\
  &\bCbar_{:,j}^{n+1} - \Psi_C(\bT_{:,j}^{n+1}) = 0, & & j=1, \dotsc, N_h. 
  \end{aligned}
\end{equation}
This is a non-linear system of equations to be solved at
each time step for the three unknowns $(\bC^{n+1},
\bT^{n+1},\bCbar^{n+1}) \in \Rn[3N_c N_h]$, defined by the function  $f: \Rn[3N_c N_h] \mapsto
\Rn[3N_c N_h]$, such that
\begin{equation}
  \label{eq:defnf}
  f   \begin{pmatrix} \bC \\ \bT \\  \bCbar \end{pmatrix} =
 \begin{pmatrix}
    (A \otimes I) \bC + (M \otimes I) \bCbar - \mathbf{b}^n \\
      \bT - \bC - \bCbar \\
      \bCbar - \Psi_C(\bT)    
  \end{pmatrix} = 0
\end{equation}

\subsection{Solution of the coupled system by Sequential Iterative Approach}
\label{sec:sia}

The most classical method for solving the coupled
problem~\eqref{eq:defnf} is the Sequential Iterative Approach (SIA) that
consists of separately solving the transport equations and the
chemical equations
cf.~\cite{carrmosbeh:04,lagneau:10,saalayorcarr98,yehtrip}. 

It will be convenient to extend the solution operator for transport
$\psi_T$ (defined in~\eqref{eq:tranpsol}) to operate on global vectors 
\begin{equation}
  \begin{aligned} \Psi_T: \quad & \Rn[N_c \times N_h] \times \Rn[N_c \times N_h] \mapsto
\Rn[N_c \times N_h] \\ &(\bC, \bQ) \to \Psi_T(\bC, \bQ), \quad \text{with }
\Psi_T(\bC, \bQ)_{i,:} &= \psi_T(\bC_{i,:}, \bQ_{i,:}), \; \text{for } i =1, \dotsc, N_c.
  \end{aligned}
\end{equation}
The solution operators $\Psi_C$ and $\Psi_T$ emphasize that both
the transport and chemistry steps can be seen as black boxes. This
is the basis for the approach followed
in~\cite{jaradreuzycoche:compgeo17}, where a flexible code that allow
switching the transport and chemistry components is presented. We show in the next
section that this black box approach is not restricted to the SIA.

At each time step, we iterate on the transport and chemistry
problem. More precisely, for each iteration $k$, we first
 solve the transport equations: 
\begin{equation}
  \bC^{n+1,k+1} = \Psi_T\left(\bC^n, M \dfrac{(\bCbar^{n+1,k} -
      \bCbar^n)}{\Delta t}\right) = -({A}^{-1} \otimes I)
\left( (M \otimes I) \bar{\bC}^{n+1, k} - \mathbf{b}^n \right)
\end{equation}
for $\bC^{n+1,k+1}$ (there is one transport equation for each
species). The new mobile total concentrations is  added
to the immobile concentrations $\bCbar^{n+1,k}$ (from the previous
iteration) to obtain $\bT^{n+1,k+1}$, and this last
concentration provides the data for the local chemical
problems
\begin{equation}\label{ch}\bCbar^{n+1,k+1}=\Psi_C(\bT^{n+1,k+1}). \end{equation}
The iterations are stopped when the difference between the solutions in the
iterations $k$ and $k+1$ is small enough, relative to a specified tolerance
 \begin{equation}\label{co.12}
\dfrac{\parallel \bC^{n+1,k+1}-\bC^{n+1,k}\parallel}{\parallel
  \bC^{n+1,k+1}\parallel}+\dfrac{\parallel \bCbar^{n+1,k+1}-\bCbar^{n+1,k}\parallel}{\parallel
  \bCbar^{n+1,k+1}\parallel}<\epsilon
\end{equation}

\MK{It is known that} in order both to control the errors due to the splitting, and to
ensure convergence,  this method requires a small time step~\cite{McQuarrieMayer:ESR05,sacaay:00,WRCR:WRCR5757}.

\subsection{Solution of the non-linear system by a Newton-Krylov
  method}

In the geochemical literature, the SIA method is known as an operator splitting approach, but it
is more properly a block Gauss--Seidel method on the coupled
system. This suggests that other, potentially more efficient, methods
could be used to solve system~\eqref{eq:defnf}. A natural candidate is
Newton's method. Since the system is large, and also because it
involves implicitly defined functions, we turn to the Newton-Krylov
variant. 

Recall that at each step of the ``pure'' form of Newton's method
for solving $f(Z)=0$, one should compute the Jacobian matrix $J
=f'(Z^k)$, solve the linear system
\begin{equation}
  \label{eq:newtres}
  J\, \delta Z = - f(Z^k),
\end{equation}
usually by Gaussian elimination, and then set $Z^{k+1} = Z^k + \delta Z$. In practice, one should use
some form of globalization procedure in order to ensure convergence
from an arbitrary starting point. If a line search is used, the last
step should be replaced by $Z^{k+1} = \delta Z + \lambda Z^k$, where
$\lambda$ is determined by the line search procedure~\cite{Kelley95}. 

The Newton--Krylov method (see~\cite{Kelley95,knollkeyes04}
and~\cite{hammvallicht05}, to which our work is closely related) is a
variant of Newton's method where the linear system~\eqref{eq:newtres} that arises at each
step of Newton's method is solved by an \emph{iterative} method (of
Krylov type), for instance GMRES~\cite{sasc:86}. 
As the linear system is not solved exactly, the convergence theory for
Newton's method does not apply directly. However, the theory has been
extended to the class of Inexact Newton methods, of which the
Newton--Krylov methods are representatives. The theory leads to a practical
consequence, by giving specific strategies for the forcing term, that
is the tolerance within which the linear system has to be solved (see~\cite{Kelley95} or the
short discussion in~\cite{compgeo:10b}).

The main advantage of
this type of method is that the full (potentially very large) Jacobian
is not needed, one just needs to be able to compute the product of the
Jacobian with a vector. As this Jacobian matrix vector product is a
directional derivative, this leads to Jacobian free methods, where
this product is approximated by finite differences. However, for the
system considered here, this has several drawbacks: in addition to its
inherent inaccuracy, computing the derivative by finite difference
entails solving one more chemistry problem. It was shown in~\cite{compgeo:10b} how this
computation could be performed exactly. This is both cheaper and more
accurate. Moreover, the above reference also shows how the natural
block structure present in the coupled system~\eqref{eq:coupleddisc} 
can be exploited to efficiently compute the residual. This is what is
done in the present work, and we now give some details.

To compute the Jacobian matrix times vector product, start from
equation~\eqref{eq:defnf}. The Jacobian matrix of $f$ has the
following block form
\begin{equation}
  \label{eq:jacf}
J_f = 
\begin{pmatrix}
  A \otimes I  & 0 & M \otimes I \\
  -I & I & -I\\
  0 & - J_C & I 
\end{pmatrix},
\end{equation}
where $J_C$ is a block diagonal matrix, whose $j$th block is
$\psi_C'(T_{:j})$, for $j=1, \dotsc, N_h$,  
and the action of the Jacobian on a vector $ v = \left(v_C^T,  v_T^T, v_{\bar{C}}^T  \right)^T$
can be computed as
\begin{equation}
\label{eq:jacvec}
  J_f \begin{pmatrix} v_C \\ v_T \\ v_{\bar{C}} \\ \end{pmatrix} = 
  \begin{pmatrix}
    (A \otimes I) v_C + (M \otimes I) v_{\bar{C}} \\[1.2ex] 
    -v_C + v_T - v_{\bar{C}} \\[1.2ex]
    v_{\bar{C}} - J_C  v_T 
  \end{pmatrix}.
\end{equation}
Of course, the Kronecker product is just a notational device, and the
matrices $A \otimes I$ or $ M \otimes I)$ are never actually
formed. Instead, we use a well know property of the Kronecker product
(see~\cite{golubvanloan4:13}) : 
\begin{equation}
\label{eq:kronprod}
(A \otimes I) \vecop(V)=  \vecop(V A^T ),
\end{equation}
and this  just requires to multiply the matrix $A$ by the
concentration vector for each species.

Since the Krylov solvers can stagnate, resulting in slow convergence,
possible strategies for preconditioning the linear system  will
be investigated in the next section. 

\section{Non linear and linear preconditioning}
\label{sec:precond}

Since the Jacobian matrix is not explicitly computed (it will just be
used as a theoretical device in what follows), the only
available options for preconditioning the system are those that
respect the block structure of the matrix. We introduce two
preconditioners derived from classical block-iterative methods, and we
show that these methods have strong links to the SIA method from
section~\ref{sec:sia}, and also to a non-linear elimination method, to
be introduced in section~\ref{nonlinearprec}.

\subsection{Block preconditioning for the coupled system}
\label{sec:linearprec}

Using a preconditioner means that instead of solving, at each Newton
iteration, the system
\begin{equation}
  \label{eq:jacob}
J_f \delta Z = -f(Z^k)  
\end{equation}
one solves one of the (mathematically equivalent) systems:
\begin{description}
\item[right preconditioning] $ J_f P^{-1} \delta y = -f(Z^k)$, and
  then $\delta Z^k= P ^{-1} \delta y$, 
\item[left preconditioning] $P^{-1} J_f \delta Z = -P^{-1} f(Z^k)$. 
\end{description}
The matrix $P$ is called a preconditioner, and it should be chosen so as to
fulfill the often conflicting goals :
\begin{itemize}
\item $P$ should be close to $J_f$ so that GMRES converges faster for
  the preconditioned system than for the original one,
\item $P$ is ``easy'' to invert, so that each iteration is not much
  more expensive than an iteration for $J_f$ alone.  
\end{itemize}

As already mentioned, in the context of the coupled
problem~\eqref{eq:defnf}, the entries of the Jacobian matrix $J_f$ are not known
explicitly, so that methods that depend on the entries of the matrix
(such as incomplete factorization methods) cannot be used. On the
other hand, we know the block structure of $J_f$
(see~\eqref{eq:jacf}), and we can invert the upper left diagonal
block (\MK{of course, when we write ``invert''  we really mean ``solve a
  linear system with''. We never actually form the inverse. See also
  remark~\ref{rem:inv1d} below}). These two facts can be exploited by resorting to block
preconditioners that respect the bock structure of the Jacobian.
Specifically, we investigate block methods derived from classical
iterative methods, namely Jacobi and Gauss--Seidel. 

\begin{description}
\item[Block Jacobi preconditioning] The preconditioning matrix ${\bf
    P}$ for  block Jacobi is:
$${\bf P_{BJ}}=
\begin{pmatrix}
A \otimes I & 0& 0\\
0 & I & 0\\
 0 & 0 & I 
\end{pmatrix},
$$
so that the action of the left-preconditioned matrix $\bf P_{BJ} J_f$ on a vector $v
=(v_C^T, v_T^T, v_{\bar{C}}^T )^T$ is :
\begin{equation}
  \label{prod:defPJv}
{\bf P_{BJ}^{-1} J_f v}=
\begin{pmatrix}
v_C + ((A^{-1} M) \otimes I) v_{\bar{C}}\\
 v_C + v_T -v_{\bar{C}}\\
- J_C v_T 
\end{pmatrix},
\end{equation}

\item[Block Gauss---Seidel preconditioning]
Here, the preconditioning matrix ${\bf P_{BGS}}$ for  block Gauss---Seidel and its
inverse are :
\begin{equation*}
{\bf P_{BGS}}=
\begin{pmatrix}
A \otimes I& 0& 0\\
 -I & I & 0\\
 0 & - J_C & I 
\end{pmatrix},
\qquad
{\bf P_{BGS}^{-1}}=
\begin{pmatrix}
A^{-1} \otimes I & 0& 0\\
A^{-1} \otimes I & I & 0\\
J_C (A^{-1} \otimes I) & J_C & I 
\end{pmatrix},
\end{equation*}
so that the action of the left-preconditioned matrix ${\bf P_{BGS} J_f}$  on a vector $v
=(v_C^T, v_T^T, v_{\bar{C}}^T)^T$ is :
\begin{equation}
  \label{prod:defPGSv}
{\bf P_{BGS}^{-1} J_f v}=
\begin{pmatrix}
v_C + \phantom{J_c v_{\bar{C}}-J_C(}  \left((A^{-1}M) \otimes I \right) v_{\bar{C}}\\
\MK{v_T - v_{\bar{C}}\phantom{J_C}} + \phantom{J_C } \left( (A^{-1}M) \otimes I \right) v_{\bar{C}}\\
v_{\bar{C}} -  J_C v_{\bar{C}} + J_C \left( (A^{-1}M) \otimes I \right) v_{\bar{C}}.
\end{pmatrix},
\end{equation}

\end{description}

Here again, as in~\eqref{eq:jacvec}, neither $A^{-1}$ nor the
Kronecker products are computed. Rather, to compute $w = \left((A^{-1}M)
  \otimes I \right) v_{\bar{C}}$, for a given vector $v_{\bar{C}} \in
\Rn[N_c N_h]$, one defines $W \in \Rn[N_x \times N_c]$ as the solution
of
\begin{equation}
  \label{eq:awmvc}
A W = M V_{\bar{C}}^T  
\end{equation}
where $v_{\bar{C}} =\vecop(V_{\bar{C}})$, then let $w=\vecop(W)$.

This means that the action of the preconditioner can be computed by
solving a transport step for each chemical species, and a
multiplication by the Jacobian of the chemical operator (which in
turns requires solving a linearized chemical problem for each grid
cell).  These are also the building blocks for the SIA formulation,
and this shows that the fully coupled approach can be implemented at
roughly the same cost per iteration as the SIA approach.

\begin{remark}
\label{rem:inv1d}
In both cases, notice that the preconditioning step involves the inverse of the
transport block. In our 1D case, this is not a difficulty, as this
block can be easily inverted. When we move to 2D or 3D problems, this
step would become more problematic. However, the preconditioners could
then be defined by replacing the matrix $A$ in the definition of $J_f$
by the application of a (spectrally equivalent) matrix, such as
several iterations of a multigrid solver.  A multigrid method such as
that proposed recently in~\cite{2017arXiv170405001S} would be
particularly appropriate. 
\end{remark}

\begin{remark}
In both cases (exact or inexact transport solver), one should expect
mesh independent convergence, as the operator being solved by GMRES
becomes bounded independently of the mesh. This expectation will be confirmed
by the numerical results in section~\ref{sec:numprec}. 
\end{remark}

\subsection{Elimination of unknown as non-linear preconditioning}
\label{nonlinearprec}

In this section, we consider alternative solution strategies, based on
eliminating some of the unknowns from the system~\eqref{eq:defnf}. We group these strategies under the
heading of ``non-linear preconditioning'' because the elimination is
done before the linearization, but it must be noted that we take
advantage of the fact that the first (block) equation
of~\eqref{eq:defnf} is linear and can be solved species by species. In
this context, this strategy should be reminiscent of the ``non-linear
preconditioning'' as introduced by Cai and Keyes~\cite{caikeyes:02},
where block of unknowns are eliminated locally.\MK{ It also provides
   an extension to multi-component transport of previous work by Kern
   and Taakili~\cite{kerntaak:10} that deals with preconditioning a
   model with one species undergoing sorption.}

We first eliminate $\bT$ from the original system~\eqref{eq:defnf}, leading to a
system with only $\bC$ and $\bCbar$ as unknowns
$$
g   \begin{pmatrix} \bC \\  \bCbar \end{pmatrix} = 0
$$
with
\begin{equation}
  \label{eq:defng}
    g   \begin{pmatrix} \bC \\  \bCbar \end{pmatrix} =
 \begin{pmatrix}
    (A \otimes I) \bC + (M \otimes I) \bCbar - \mathbf{b}^n \\
      \bCbar - \Psi_C(\bC + \bCbar)    
  \end{pmatrix}.
\end{equation}

The Jacobian of the new system is
$$J_g = 
\begin{pmatrix}
  A \otimes I & M \otimes I \\
  -J_C & I - J_C 
\end{pmatrix}.
$$
and the Jacobian by  vector product is :
\begin{equation}
  \label{prod:defjgv}
J_g \begin{pmatrix}
v_C \\
v_{\bar{C}}
\end{pmatrix}= 
\begin{pmatrix}
  (A \otimes I) v_C + (M \otimes I) v_{\bar{C}} \\
  - J_C (v_C+v_{\bar{C}}) + v_{\bar{C}}
\end{pmatrix},
\end{equation}

\medskip
One can even go one step further, by eliminating both the unknowns $\bT$ 
and $\bC$, to obtain a system with $\bCbar$ as the single unknown: 
\begin{equation}
\label{eq:hsystem}
h(\bCbar) =  \bCbar -\Psi_C \biggl( \Psi_T\left( \bC^n, \MK{M} \dfrac{\bCbar -
    \bCbar^n}{\Delta t} \right) +\bCbar  \biggr) =0.
\end{equation}
This equation is presented in fixed point form,  and solving it by
fixed point method recovers the SIA method described in
Section~\ref{sec:sia}. But equation~\eqref{eq:hsystem} can also be solved by Newton's
method. We presented a Jacobian-free Newton-Krylov method
in~\cite{compgeo:10b}. Its main advantage is to require that one be
able to evaluate the nonlinear residual function $h$, and the Jacobian
times vector product
is approximated by finite difference. Since it involves solving an additional
chemical problem, this step is very expensive. 
 In this work, we use the exact Jacobian of $h$, which requires that
 we ``open the black box'', to look in more detail at the structure
 of the Jacobian. To do this, we rewrite the equation above by making
 use of the precise definition of the transport operator $\Psi_T$: 
\begin{equation}
  \label{eq:defh} 
h(\bCbar) = \bCbar -\Psi_C \biggl( (A^{-1} \otimes I) \left(\mathbf{b}^n
  -(M \otimes I) \bCbar \right)+ \bCbar \biggr)  = 0.  
\end{equation}
If we denote by $J_T = I- \left((A^{-1} M) \otimes I \right)$ the
Jacobian matrix of $\Psi_T$, we easily see
that the Jacobian by vector product for $h$ is:
\begin{equation}
  \label{prod:defjhv}
J_hv= v - J_C J_T v = v - J_C v \MK{+} J_C \left( (A^{-1} M) \otimes I \right) v.
\end{equation}
Both the alternative formulation and its Jacobian involve the
resolution of transport at each Newton step. One again expects that
convergence \MK{of the linear solver will become} independent of the mesh
size~\cite{kerntaak:10}. This will be confirmed in
Section~\ref{sec:numres}.

\subsection{Links between block preconditioning and elimination}

The main advantage of the SIA approach is that it allows the reuse of
existing software modules for solving transport and chemistry. This is
an important practical issue, as these modules may have been developed
independently, even by different groups. This advantage is offset by
the possibly slower convergence when compared to GIA, see for
example~\cite{hammvallicht05}). The ``$h$'' method proposed in
section~\ref{nonlinearprec} (keeping only $\bar{\bC}$ as unknown)
\MK{belongs to the GIA family, in the sense that it solves both
  chemistry and transport as a coupled system}. This method is in the
same spirit as the one proposed by Knabner and his
group~\cite{Hoffmann2009,knabner05,knabner07}. The reduction method
proposed in these papers is more general than the formulation in
section~\ref{sec:formul}, but their ``resolution function'' is
identical to the chemical solution operator $\Psi_C$. 

However, the method aims at keeping some of the advantages of SIA: it
allows to keep transport and chemistry as separate
modules. Equation~\eqref{eq:defh} shows that the evaluation of the
residual when solving the system with $h$ requires solving a transport
step for each species, and then solving a local chemical equilibrium
system at each grid cell. \MK{It was shown
  in~\cite{compgeo:10b,knollkeyes04} that this is sufficient, provided
  one accepts to approximate the Jacobian matrix by vector product by
  a finite difference quotient. However, the present paper makes the
  point that it is both cheaper and more accurate to compute this
  matrix-vector product exactly. This requires more cooperation from
  the chemical solver, as one needs to access its Jacobian
  computation};

\MK{We also give some indications for comparing the cost of the
  ``$h$'' method with SIA. At each iteration of SIA, one has to solve
  one transport problem per component, and one chemical equilibrium
  problem per grid cell. For the ``$h$'' method, the same cost is
  incurred at each Newton iteration, and one has to add the cost of
  the Jacobian matrix vector product at each GMRES iteration. This
  translates to one additional transport problem for each component,
  and one \emph{linearized} chemical problem for each cell. We neglect
  the overhead of constructing the Arnoldi basis in GMRES, as the
  number of iterations is expected to be small (as confirmed by the
  results in section~\ref{sec:numprec}). So the number of transport
  problem to be solved is the sum of the number of Newton and GMRES
  iterations, but the number of chemical equilibrium problems is only
  the number of Newton iterations. One can hope (and again this is
  confirmed by our numerical experiments) that the number of Newton
  iterations will remain small. Additionally, it has been observed
  that the SIA method needs a small time step to reduce the splitting
  errors (see~\cite{Momascompar,lagneau:10}). 

It is more difficult to comment on the cost of GIA, as this will
depend in a crucial way on the efficiency of the solver. With proper
care and effort, DSA methods can be made very efficient~\cite{hammond:14}. The
results in~\cite{Hoffmann2009} also show that a method based on a
structure similar to that of the ``$h$'' method was among the fastest
on the MoMaS benchmark.
}

\medskip
It is interesting to note that both the SIA method and
the ``$h$'' function method from section~\ref{nonlinearprec} can be
interpreted in terms of the block preconditioners introduced in
section~\ref{sec:linearprec}. 

Indeed, the elimination method can be interpreted as a (linear) change
of variables, given by 
\begin{equation}
  \label{eq:trans2}
  \left\{ 
    \begin{aligned}
      \tilde{\bC} &= \bC & &+ \hphantom{I-}  \left((A^{-1} M)\otimes I  \right)\; \bar{\bC} \\
      \tilde{\bT} &= & \bT &- \left( (I - A^{-1} M)\otimes I\right)  \bar{\bC}
      \\
      \tilde{\bCbar} &= & & \hphantom{- \left( (I - A^{-1} M)\otimes
          I\right)\; } \bar{\bC}
    \end{aligned}
\right. ,
\end{equation}
whose matrix
\begin{equation*}
\quad B = 
\begin{pmatrix}
      I & 0 & \hphantom{-I-\;} (A^{-1} M) \otimes I   \\
      0 & I &  - ( I - A^{-1}M) \otimes I    \\
      0 & 0 & \hphantom{- ( I - A^{-1}M) \otimes} I
\end{pmatrix}
\end{equation*}
 is block triangular, so that the transformation is easily inverted. 
The transformed system
\begin{equation*}
\tilde{f}\begin{pmatrix} \tilde{\bC} \\ \tilde{\bT} \\  \tilde{\bCbar} \end{pmatrix}  = f \begin{pmatrix} \bC \\ \bT \\  \bCbar \end{pmatrix}
\end{equation*}
 takes the simple form
\begin{equation}
\label{eq:systrans}
\tilde{f} \begin{pmatrix} \tilde{\bC} \\ \tilde{\bT}
  \\  \tilde{\bCbar} \end{pmatrix} =
\begin{pmatrix}
(  A \otimes I) \tilde{\bC} - \mathbf{b} \\ 
-\tilde{\bC} + \tilde{\bT}  \\
\tilde{\bCbar} - \Psi_C \biggl( \tilde{\bT} + \left( (I - A^{-1} M) \otimes I
  \right) \tilde{\bCbar} \biggr)  
\end{pmatrix}=0.
\end{equation}
and the last equation is obviously identical to~\eqref{eq:defh}, after
having eliminated the  first two unknowns.

Because the transformation used is linear, and because Newton's method
is invariant under a linear transformation, the iterates between the
original and the transformed system will be related by  the same
transformation (provided the initial guesses are related similarly). Note that this
will \emph{not} necessarily be true for an inexact Newton's method, such as the
Newton-Krylov method used in this work. However,
Deuflhard points out~\cite[sec. 2.2.4]{deuflhard:11} that because the
Newton residual is invariant under affine
transformations, GMRES is the natural choice for
solving the linear system arising at each Newton iteration. Moreover,  left
preconditioning can be used with GMRES provided the preconditioning
matrix 
\begin{itemize}
\item is kept constant throughout the Newton iterations, 
\item and is incorporated in the convergence monitoring criteria.
\end{itemize}

Note that the Jacobian of the transformed system is the block
triangular matrix
\begin{equation}
  \label{eq:jacobg}
  J_{\tilde{f}} =
  \begin{pmatrix}
    A \otimes I & 0 & 0 \\
    -I & I & 0 \\
    0 & -J_{\MK{C}} & \MK{J_h}
  \end{pmatrix}
\end{equation}
\MK{where $J_h = I -J_{\MK{C}} \left((I- A^{-1} M ) \otimes I \right)$,
defined in~\eqref{prod:defjhv},} is actually 
 the Schur complement of $J_f$ with respect to
its first two variables.

This confirms the close  link between the non-linear
elimination method from section~\ref{nonlinearprec} and block
Gauss--Seidel preconditioning in~\eqref{prod:defPGSv}: as noted above, the
Jacobian of $h$ is exactly  the Schur complement of the
Jacobian of $f$. In both cases, one needs to compute $ (J_C  ((A^{-1}M)
\otimes I)) v$. Of course, we are once again taking advantage of the fact that one
of the operators is linear ! 

An interesting consequence of the change of variables is that, because
$J_{\tilde f}$ and the change of variable matrix $B$ are both block
triangular, differentiating the identity
$f = \tilde f \circ B$ leads to a block triangular factorization of
$J_f$,
\begin{equation}
  J_f = J_{\tilde f} B = 
  \begin{pmatrix}
    A \otimes I & 0 & 0 \\
    -I & I & 0 \\
    0 & -J_{\MK{C}} & S
  \end{pmatrix}
\begin{pmatrix}
      I & 0 & \hphantom{-I-\;} (A^{-1} M) \otimes I   \\
      0 & I &  - ( I - A^{-1} M) \otimes I    \\
      0 & 0 & \hphantom{- ( I - A^{-1} M) \otimes} I
\end{pmatrix}.
\end{equation}
This factorization could be used as a basis for constructing efficient
preconditioners for $J_f$, much as in the spirit
of~\cite{doi:10.1137/120883086},~\cite{doi:10.1137/S1064827500377435} (see also
~\cite{doi:10.1137/S1064827599355153}).
The matrix $B$ is upper
triangular with unit diagonal, so that all its eigenvalues are
equal to~1. As noted in~\cite{White2011}, under these conditions, Krylov subspace methods for the
preconditioned system $J_{\tilde{f}}^{-1} J_f$ converge in 1 iteration, giving
$J_{\tilde f}$ as a perfect preconditioner. 

Because it contains $S$ as its $(3,3)$ block, $J_{\tilde f}$ cannot be formed,
let alone inverted. A first solution is to replace $S$ by an
approximation $\tilde{S}$ that is easier to invert. The simplest
choice is to take $\tilde{S}=I$, which gives the block Gauss--Seidel
preconditioner from section~\ref{sec:linearprec}. 
But actually, systems with $S$ can be solved by an iterative method,
as the matrix vector product $S v$ can be computed by proceeding as
for the Gauss--Seidel preconditioner at the end of
section~\ref{sec:linearprec} (see equation~\eqref{eq:awmvc}).

The numerical results in section~\ref{sec:numprec}
(cf. figure~\ref{fig:4}) will confirm that the performance of 
SIA and block Jacobi on the one hand 
and the elimination method and block
Gauss--Seidel on the other hand
are very similar.

\section{Numerical results}
\label{sec:numres}
\subsection{MoMaS Benchmark : 1D easy advective case}
The MoMaS Benchmark has been designed to compare numerical methods for reactive transport models in 
1D and 2D. 
Different methods for coupling have been used to solve this
benchmark. The definition has been published in~\cite{compgeo:10a} 
and the results of participants are compared in the synthesis
article~\cite{Momascompar}. 

The geometry of the test case is shown in Figure~\ref{fig:1}. For the
1D test case, the domain is heterogeneous and composed of two porous media
A and B. Medium A is highly permeable with low porosity and low reactivity in
comparison with medium B. Their physical properties are given in
Table~\ref{tab:propmat}. The \MK{Darcy} velocity is constant over the
domain and is equal to $\phi u = 5.5 \, 10^{-3} \text{L}\text{T}^{-1}$. 

\begin{figure}[ht]
\centerline{\includegraphics[width=0.6\textwidth]{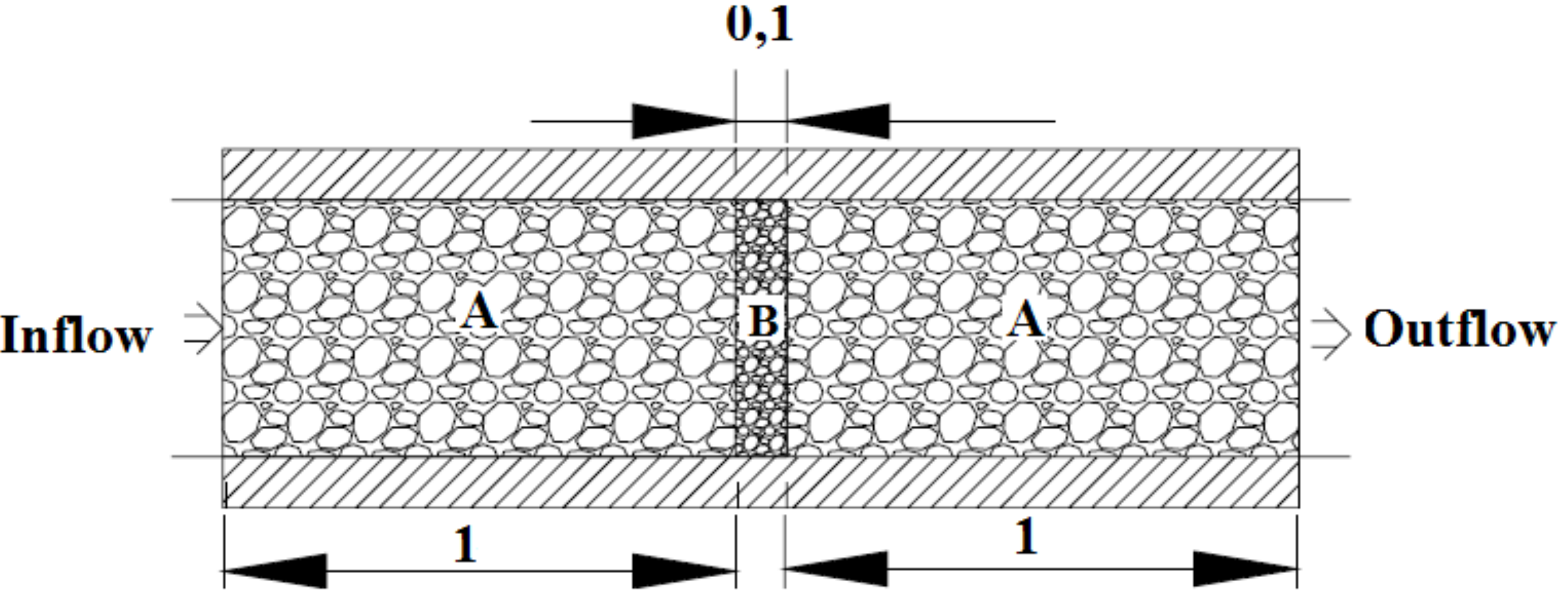}}
\caption{Geometry of the domain}
\label{fig:1}
\end{figure}

\begin{table}[htb]
  \centering
  \begin{tabular}{l|c|c} 
    & Medium A & Medium B \\ \hline
    Porosity $\phi$ ($-$) & 0.25 & 0.5 \\
    Total immobile concentration $T_S$ & 1 & 10 \\
\hline
  \end{tabular}
  \caption{Physical properties of the materials}
  \label{tab:propmat}
\end{table}

The chemical reactions are summarized in Table~\ref{tab:1}. The 7
reactions involve 9 mobile species (in the aqueous phase) and 3
immobile species (in the solid phase).

The characteristic feature of this chemical system is that it contains
large stoichiometric coefficients 
that range from -4 to 4 and a large variation of equilibrium constants
between $10^{-12}$ and $10^{35}$. 

Boundary and initial conditions are presented in
Table~\ref{tab:2}. The simulation involves an injection period
followed by a leaching period, so that the system is returned to its
initial state. Note that the total for the immobile species is given
in Table~\ref{tab:propmat}.

\begin{table}[htb]
\begin{minipage}[c]{0.54\textwidth}
  \begin{center}
\begin{tabular}{|c|cccc|c|c|}
\hline
 & ${ X_1}$ & ${ X_2}$ & ${ X_3}$ & ${ X_4}$ & { $S$} & { $K$}\\
\hline
${ C_1}$ & 0 & -1 & 0 & 0 & 0 & ${\bf 10^{-12}}$\\
${ C_2}$& 0 & 1 & 1 & 0 & 0 & 1\\
${ C_3}$ & 0 & -1 & 0 & 1 & 0 & 1\\
${ C_4}$ & 0 & {\bf -4} & 1 & $ 3$ & 0 & $10^{-1}$\\
${ C_5}$ & 0 & ${\bf 4}$ & $3$ & 1 & 0 & ${\bf 10^{+35}}$\\
\hline
${ CS_1}$ & 0 & 3 & 1 & 0 & 1 & $10^{+6}$\\
${ CS_2}$ & 0 & -3 & 0 & 1 & 2 & $10^{-1}$\\
\hline
\end{tabular}
\caption{Morel tableau for Chemical equilibrium}
\label{tab:1}
\end{center}
\end{minipage}
\hfill
\begin{minipage}[c]{0.44\textwidth}
\begin{center}
\begin{tabular}{|l|cccc|}
\hline
Total & $T_{1}$ & $T_{2}$ & $T_{3}$ & $T_{4}$  \\
Conc.  & 0 & -2 & 0 & 2 \\
\hline
{\bf Injection} & 0.3 & 0.3 & 0.3 & 0 \\
t$\in$[0,5000]&  &  &  & \\
\hline
{\bf Leaching}  & 0 & -2 & 0 & 2 \\
t$\in$[5000,..]&  &  &  & \\
\hline
\end{tabular}
\caption{Boundary conditions}
\label{tab:2}
\end{center}
\end{minipage}
\end{table}

\subsubsection{Sample results}

The results obtained in~\cite{Momascompar} indicate that one needs to refine
the mesh around medium B if one is to obtain accurate results. For all
test cases, we use a mesh such that $h_A = 4 h_B$. 
\MK{The computations were carried out with the various methods
  presented in this paper. When used with the appropriate numerical
  parameters they all gave comparable results. The figures in this
  section were obtained with the $h$-method.
}

Figure~\ref{fig:2x} shows profiles of the concentrations of several
species. The left and middle images show concentrations of (a part of)
mobile species $C1$ and immobile species $S$ at an early time $t=10$,
whereas the right image shows concentration of aqueous species $C_2$
on a smaller interval during the leaching period, at $t=5010$. The
middle image highlights the effect of the heterogeneity at $x=1$.

\begin{figure}[htbp]
\begin{center}
\includegraphics[width=\textwidth]{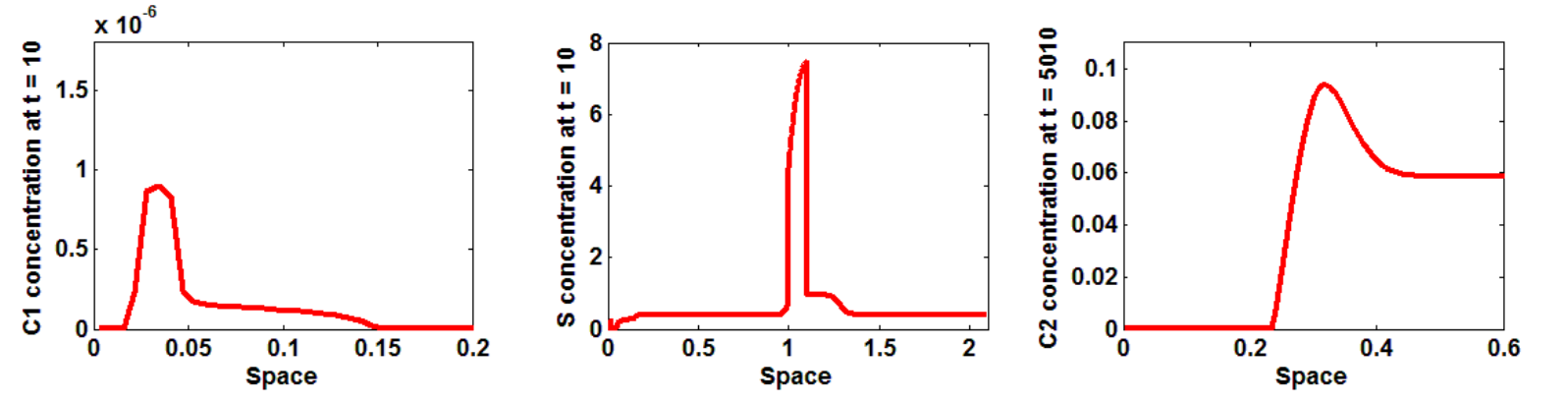}
\caption{Concentrations profiles. Left figure: $C_1$ at $t=10$, middle
figure $S$ at $t=10$, right figure $C_2$ at $t=5010$.}
\label{fig:2x}
\end{center}
\end{figure}

Figure~\ref{fig:2t} shows elution curves, that is evolution of the
concentrations at the end of the domain ($x=2.1$) as a function of
time, for a mesh with $384$ elements. Notice that the evolution of
$C5$ follows a fairly complex pattern.  It turns out that the accuracy
for both species $X3$ and $C5$ is quite sensitive to the mesh size
used. We come back to this point in detail in the next subsection.

\begin{figure}[htbp]
\begin{center}
\includegraphics[width=\textwidth]{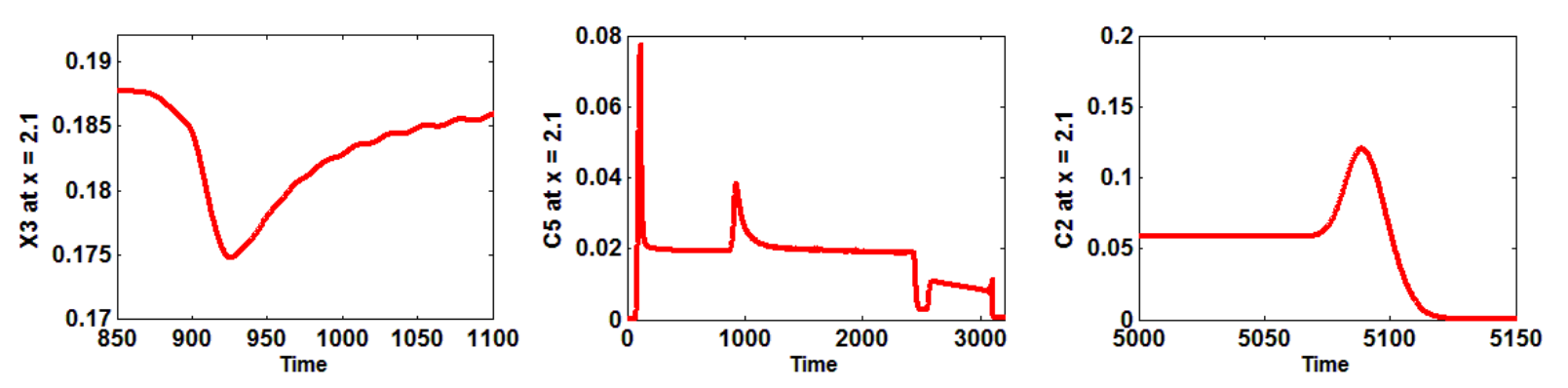}
\caption{Elution curves (concentrations at right end of the domain as
  a function of time). Left image: species $X3$, middle image:
  species $C5$, right image: species $C2$}
\label{fig:2t}
\end{center}
\end{figure}

The results shown on these figures are in good agreement with those
showed by various groups in the comparison paper\cite{Momascompar}. 



As mentioned above, some of the species show a very sensitive
dependence on the mesh and time step used. It has proven necessary to
use very fine meshes, as well as small time steps, in order to resolve
these species accurately. \MK{We now study in more detail how the
  accuracy depends on the time and space meshes}

\subsubsection{Influence of spatial discretization}

The evolution of species $X3$ and $C5$, as a function of time,
exhibits unphysical oscillations. The origin of these oscillations has
been explained in~\cite{lagneau:10}, and is due to the interplay
between the very stiff reactions and the spatial discretizations. They
should decrease as the mesh is refined.

Figures~\ref{fig:9} shows that this is indeed what
happens as we increase the number of discretization points. A mesh
with 384 points gives qualitatively correct results, but we have also
made use of a finer meshes in order to obtain more accurate results. 

\begin{figure}[htbp]
\begin{center}
\includegraphics[width=0.45\textwidth]{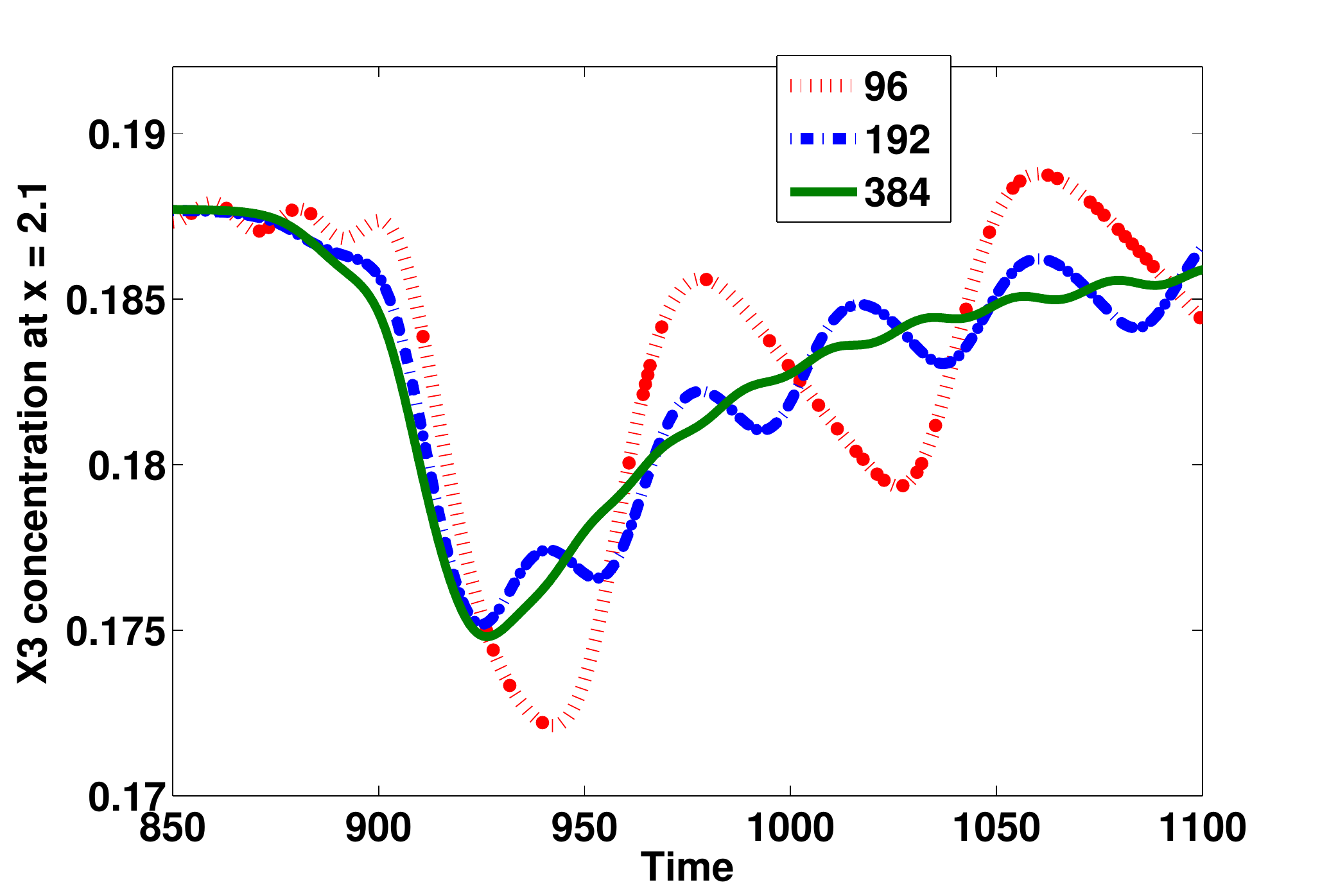} 
\includegraphics[width=0.45\textwidth]{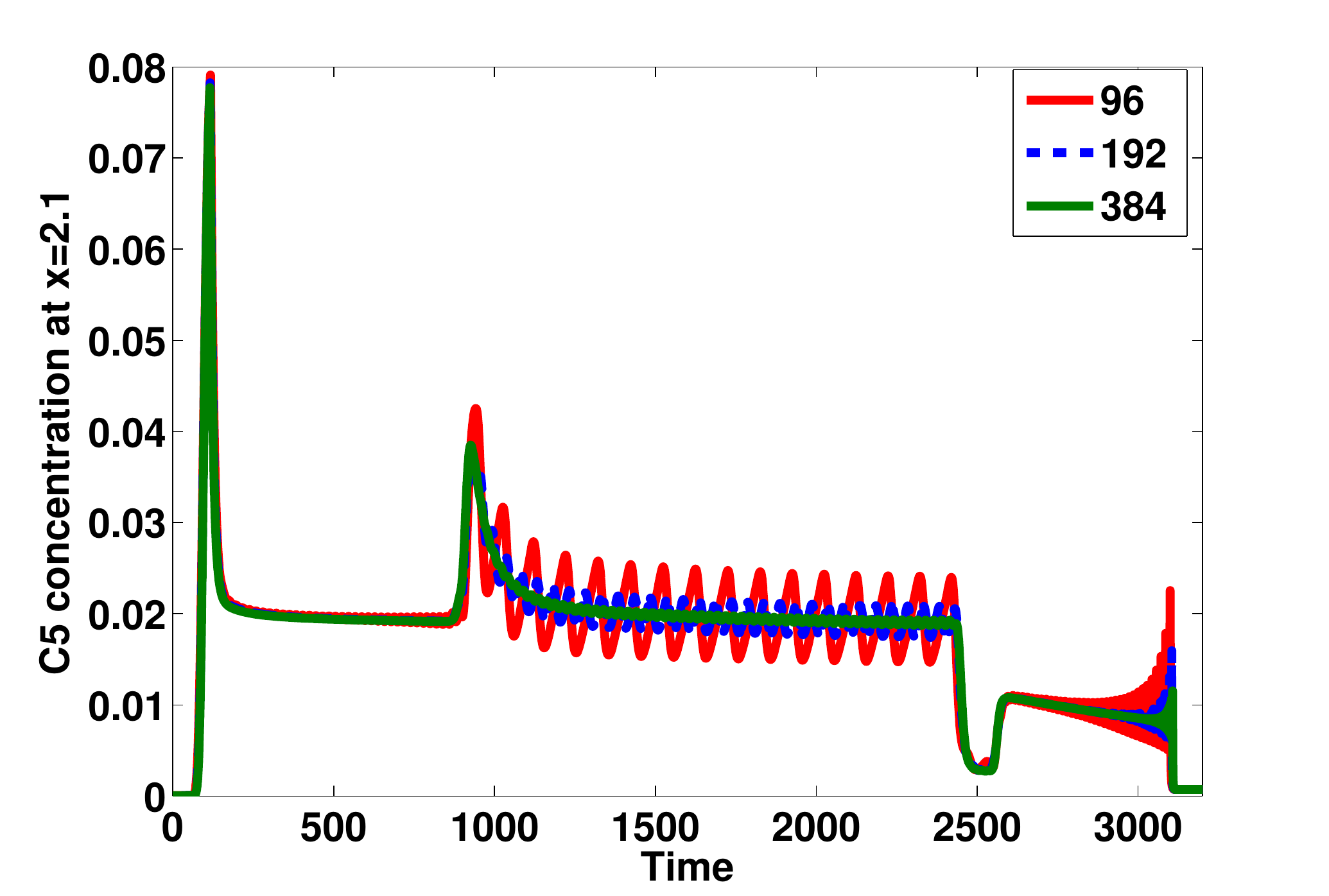} 
\caption{Elution curves of $X_3$ (left) and $C_5$ (right) concentration, for various mesh resolutions}
\label{fig:9}
\end{center}
\end{figure}

\medskip
The effect of the mesh size on the accuracy of the results is also
shown by looking in detail at two specific species: mobile species
$C_1$ and immobile species $S$, both at time $t=10$. They both exhibit a
sharp peak, and we focus on the accuracy with which the location and
amplitude of the peak can be determined. These elements were part of
the comparison criteria for the benchmark, and were examined in detail
in~\cite{Momascompar}. 

We also discuss the influence of the discretization scheme, by
comparing the three schemes introduced in section~\ref{sec:convnum}
from the point of view of accuracy. Their relative
efficiencies are compared in section~\ref{sec:CPU}. \MK{In
  figures~\ref{fig:6C1} and~\ref{fig:6S}, the curves are labeled by
the total number of discretization points in space. The coarsest mesh,
corresponding to $n=96$, has $\Delta x = 0.025$ in medium A and
$\Delta x = 0.00625$ in medium B. In both cases the time step was
chosen as $\Delta t = \Delta t_c$ (advective time step), with a
CFL=0.1, as given by the smaller space mesh in medium B. Note that
since the CFL is fixed, each curve also corresponds to a different time step.} 

\medskip 
Figure~\ref{fig:6C1} compares the concentrations of species $C_1$ on
the interval $[0, 0.2]$, at time $t=10$, as computed by the three
discretization schemes on increasingly refined meshes.
As soon as the mesh has sufficiently many points to successfully
resolve the solution, all three schemes give identical solutions. On the
other hand, one needs at least 384 points (and preferably 768) to
obtain a satisfactory solution.

\begin{figure}[htbp]
\begin{center}
\includegraphics[width=0.98\textwidth]{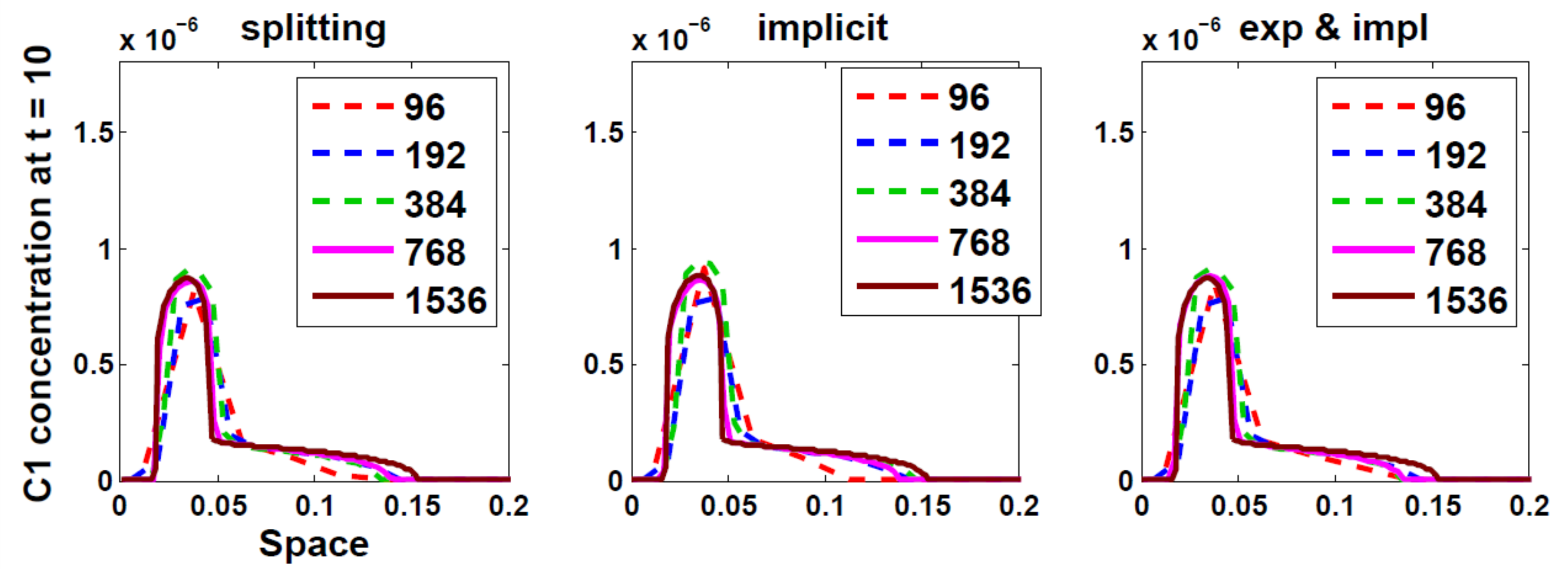}
\caption{Concentration of species $C_1$ at $t=10$ for the different
  discretization schemes and various mesh resolutions}
\label{fig:6C1}
\end{center}
\end{figure}
Figure~\ref{fig:6S} compares the concentration of species $S$ (a
sorbed species) on the interval bracketing the location of the peak
$[0, 0.15]$, at time $t=10$, as computed by the three discretization
schemes on increasingly refined meshes.  Here also, the three schemes
give identical solutions when the mesh is fine enough. This time, one
needs at least 768 mesh points to obtain a converged solution. One
should compare Figure~6 in~\cite{Momascompar}, and also refer to
Table~4 there, where the location and amplitude of the peak are
tabulated for all the methods used in the benchmark. We have obtained
values of $x=0.0175$ for the location of the peak, and $S=0.985$ for
its height. These values are in the range reported by the other teams,
but are different form the ``mean'' values as reported
in~\cite{Momascompar}.  They are however very close to the
``reference'' values found in~\cite{carray:10}.

\begin{figure}[htbp]
\begin{center}
\includegraphics[width=0.98\textwidth]{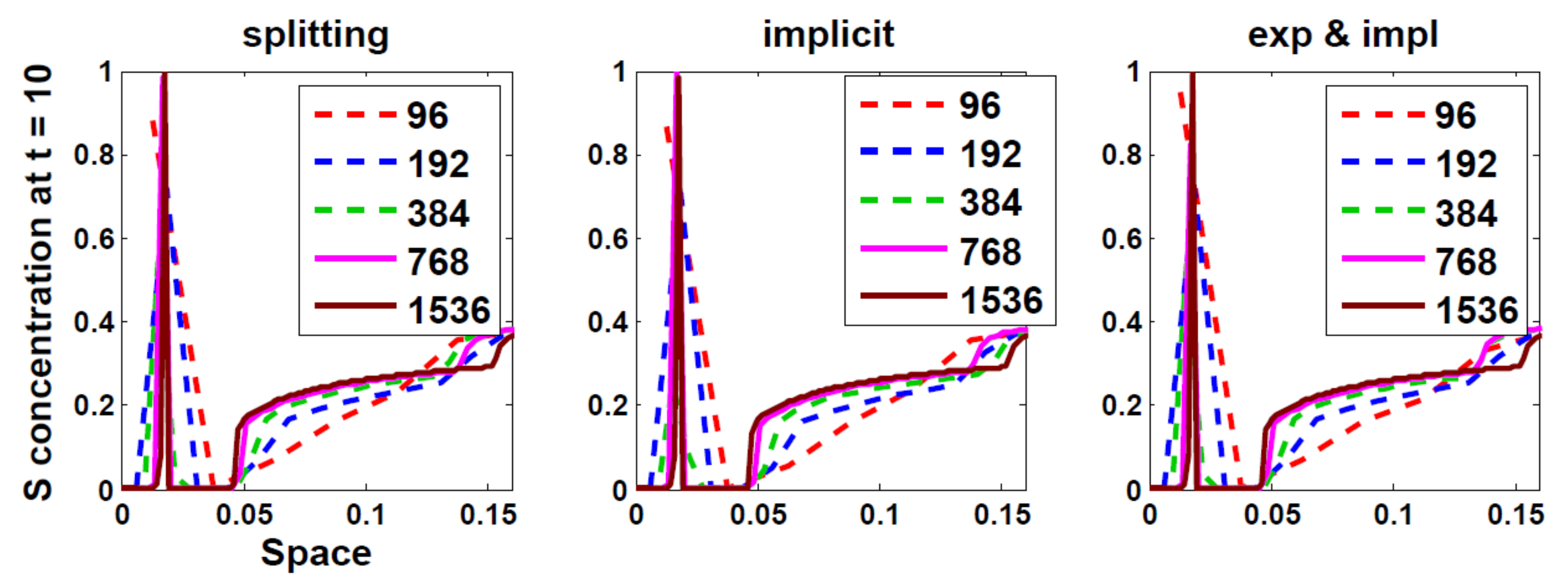}
\caption{Concentration of species $S$ at $t=10$ for the different
  discretization schemes and various mesh resolutions}
\label{fig:6S}
\end{center}
\end{figure}

\subsubsection{Influence of temporal discretization}

We now discuss the influence of the time step, or more precisely the
value of the CFL coefficient (given by $CFL=u \phi \Delta t / h_{\min}
$ on the accuracy.

The three schemes used for time discretization have different
behavior with respect to the choice of time step:
\begin{description}
\item[The splitting method] allows the use of different time steps
  (and different numerical methods) for the advection and the
  diffusion step. The advection time step is restricted by the CFL
  condition (it is an explicit sub-step), whereas the time step for
  diffusion is not restricted by stability. We have used three
  different choices: chose the same time step for advection and
  diffusion (respecting the CFL condition), once with a CFL condition
  of 1, and once with a CFL condition of $0.1$, and choose a diffusion
  time step 3 times larger than the advection time step (the latter
  chosen by the CFL condition). 
\item[The fully implicit method] here there is only one time step, and
  no stability restriction. We have compared three time steps,
  corresponding to CFL coefficients of $0.1$, $1$ and $3$ respectively.
\item[The explicit--implicit method] this method also imposes a single
  time step, and in addition it is subject to the stability condition
  of the explicit method, so the time step is restricted to $CFL \le
  1$. We compare 2 time steps, corresponding to CFL coefficients of
  $0.1$ and 1 respectively.
\end{description}

\medskip
\MK{Figure~\ref{fig:Sdt} compares} the concentration of species $S$ as
computed by the three methods, for the time step sizes chosen as explained
above, for two different mesh resolutions.

\MK{For all three methods the location of the peak was correctly
  determined even for larger values of the time step and the mesh
  size, but its amplitude was only correctly estimated for the finer
  mesh size.

For both splitting and explicit-implicit schemes, it is not necessary to 
use a CFL of 0.1 (a CFL of 1 is enough) for the peak amplitude to reach the value 1, but what is 
needed is to refine the mesh (up to 1536 nodes). However, the fully implicit scheme needs 
both a  small time step corresponding to a CFL of 0.1 and a fine mesh with
1536 nodes to obtain the same results as the other schemes.}
\begin{figure}[thbp]
  \begin{center}
\includegraphics[width=0.98\textwidth]{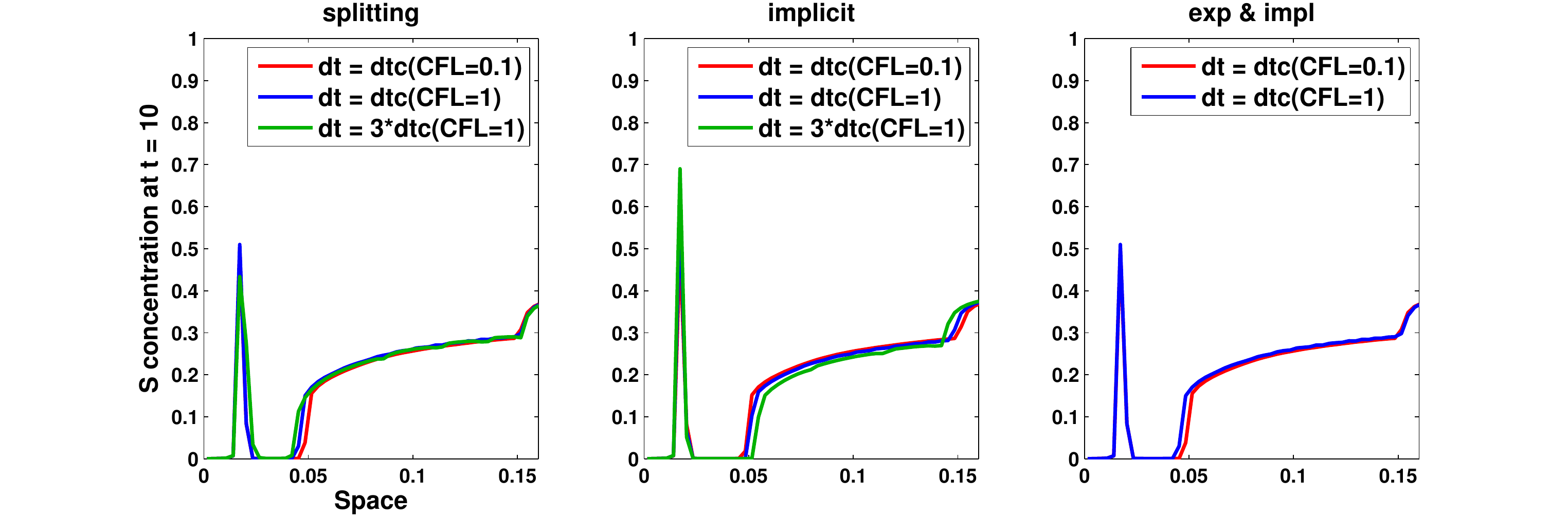} \\[1.5ex]
\includegraphics[width=0.98\textwidth]{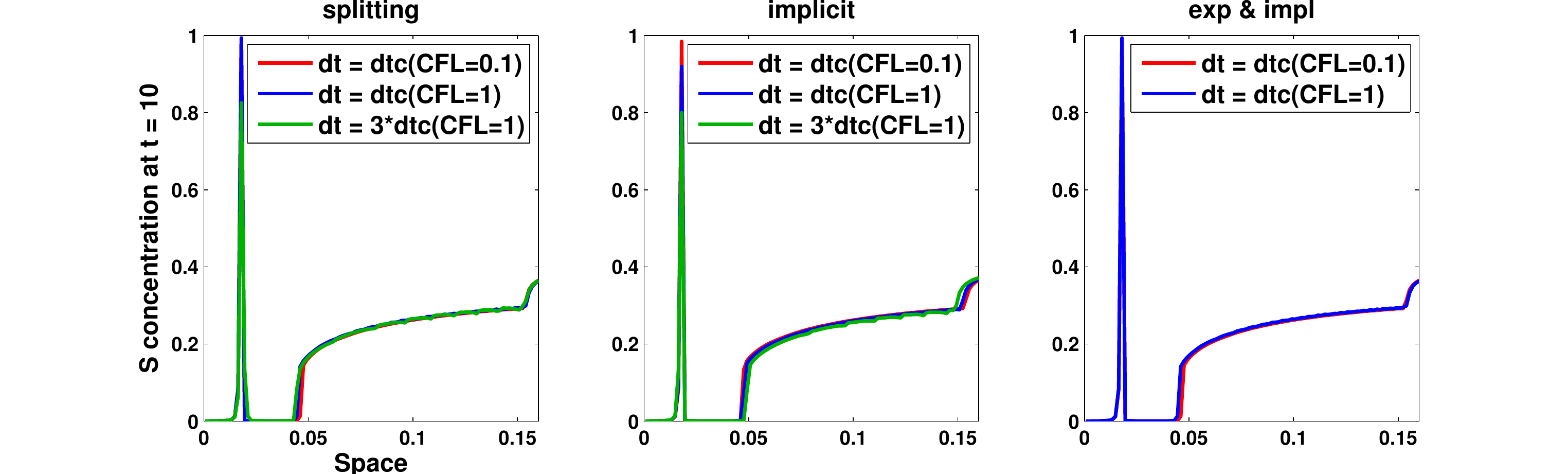}
\caption{\MK{Effect of the CFL} condition on the concentrations in $S$
  (top figure 768 nodes, bottom figure 1536 nodes)}
\label{fig:Sdt}
\end{center}
\end{figure}


\subsubsection{CPU time}
\label{sec:CPU}

We now compare the relative efficiency of the three time
discretization methods. One should keep in mind that the problem under
study is a 1D model, so that solving a linear system is a not very
time consuming. The conclusions reached below would have to be
re-examined for a 2D (and even more for a 3D !) model.  We have also
compared the effect of using the fully implicit method with a variable
time step (albeit not an adaptive choice, a sequence of pre-computed
time steps is used for different phases of the solution evolution).

\MK{
Figure~\ref{fig:11} compares the CPU times required by the
$h$-method with different time-discretization schemes as the space
mesh is refined. 
The time step was chosen as follows (in all cases, the smallest mesh
size was used):
\begin{description}
\item [For the splitting method] The diffusion time step is 3 times larger than
the advection time step that respects a CFL coefficient of 1.
\item [For the explicit-implicit method] One time step is used
  controlled by a CFL condition of 1.
\item [For the fully implicit method] One time step is used
  corresponding to a CFL coefficient of 3 (no stability restriction).
\item [For the fully implicit (variable) method] Variables time steps
  are used, as described in table~\ref{tab:cflvar}.  The time steps
  are chosen so as to have a small time step when strong variations
  happen   due to the reactions (especially during injection and leaching period) and a large
 time step is used for the intervals that represent the steady state.
\end{description}
\begin{table}[htbp]
    \centering
    \begin{tabular}{crrrrrrr}
Start time & 0 & 20 & 100 & 2500 & 3200 & 5000 & 5100  \\
End time  & 20 & 100 & 2500 & 3200 & 5000 & 5100 & 6000 \\ \hline
CFL value       &1 & 5&10 & 5& 40 &1 & 40 
    \end{tabular}
    \caption{CFL values for variable time step simulation}
    \label{tab:cflvar}
  \end{table}
}%
As expected, using a variable time step results in large
savings (while maintaining the accuracy). Among the 3 schemes with
fixed time step, the explicit--implicit method is the most expensive,
with the fully implicit and the splitting methods leading to
comparable costs. 
\begin{figure}[htbp]
  \begin{center}
\includegraphics[width=0.5\textwidth]{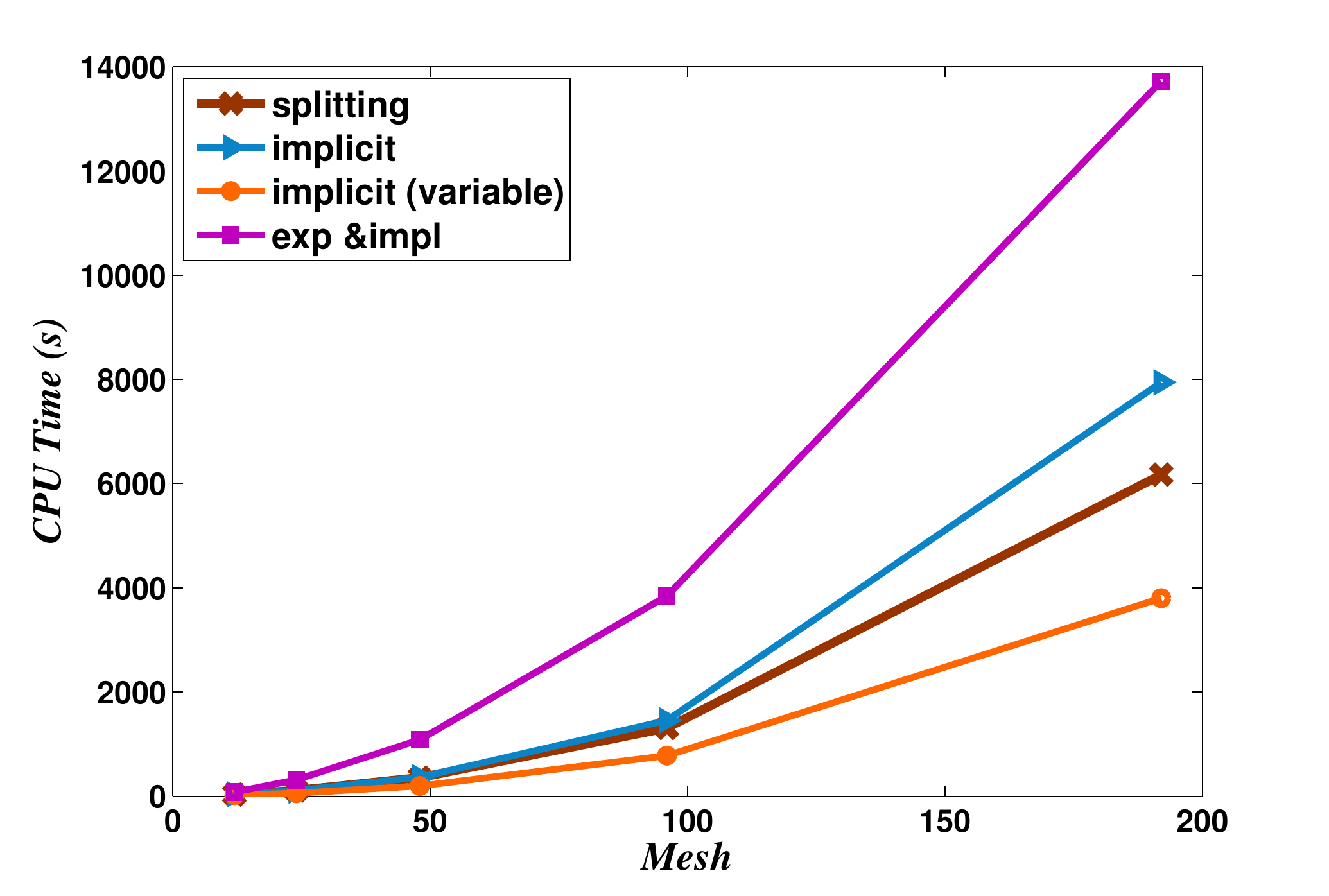}
\caption{CPU time required by the time discretization methods \MK{as the
  space mesh is refined}}
\label{fig:11}    
  \end{center}
\end{figure}

\subsubsection{Influence of preconditioning strategy} 
\label{sec:numprec}

In this section, we compare the various preconditioning strategies
discussed in section~\ref{sec:precond}. Our main criteria will be the
number of linear and non-linear iterations. We have not tried to
optimize the inexact Newton strategy, but have just relied on the
default choices as provided in the Newton--Krylov code (we use the
\texttt{nsoli} code from the book by
C. T. Kelley~\cite{Kelley03}). \MK{In all experiments, GMRES was used
  without restart, and with a maximum number of allowed iterations
  fixed at 40}.  

Figure~\ref{fig:4} shows how  these numbers change as the mesh is refined. The
various linear preconditioning strategies (applied to the coupled
system) are compared with the elimination, or nonlinear
preconditioning, strategy. As predicted, the nonlinear elimination
strategy has the smallest number both for non-linear and for linear
iterations. It also shows a behavior that is independent of the mesh
size. 
\begin{figure}[bhtp]
\begin{center}
\includegraphics[width=0.45\textwidth]{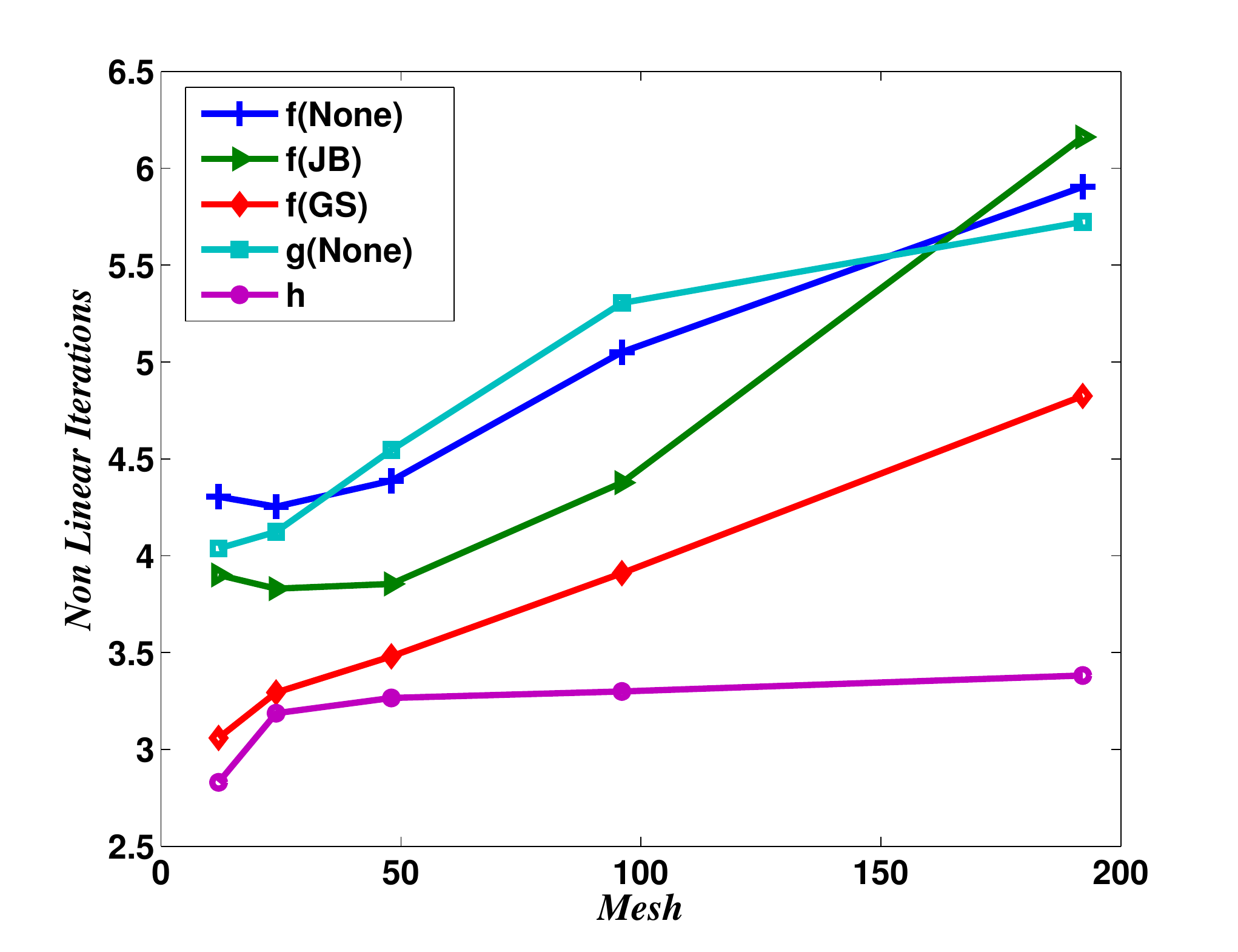} 
\includegraphics[width=0.45\textwidth]{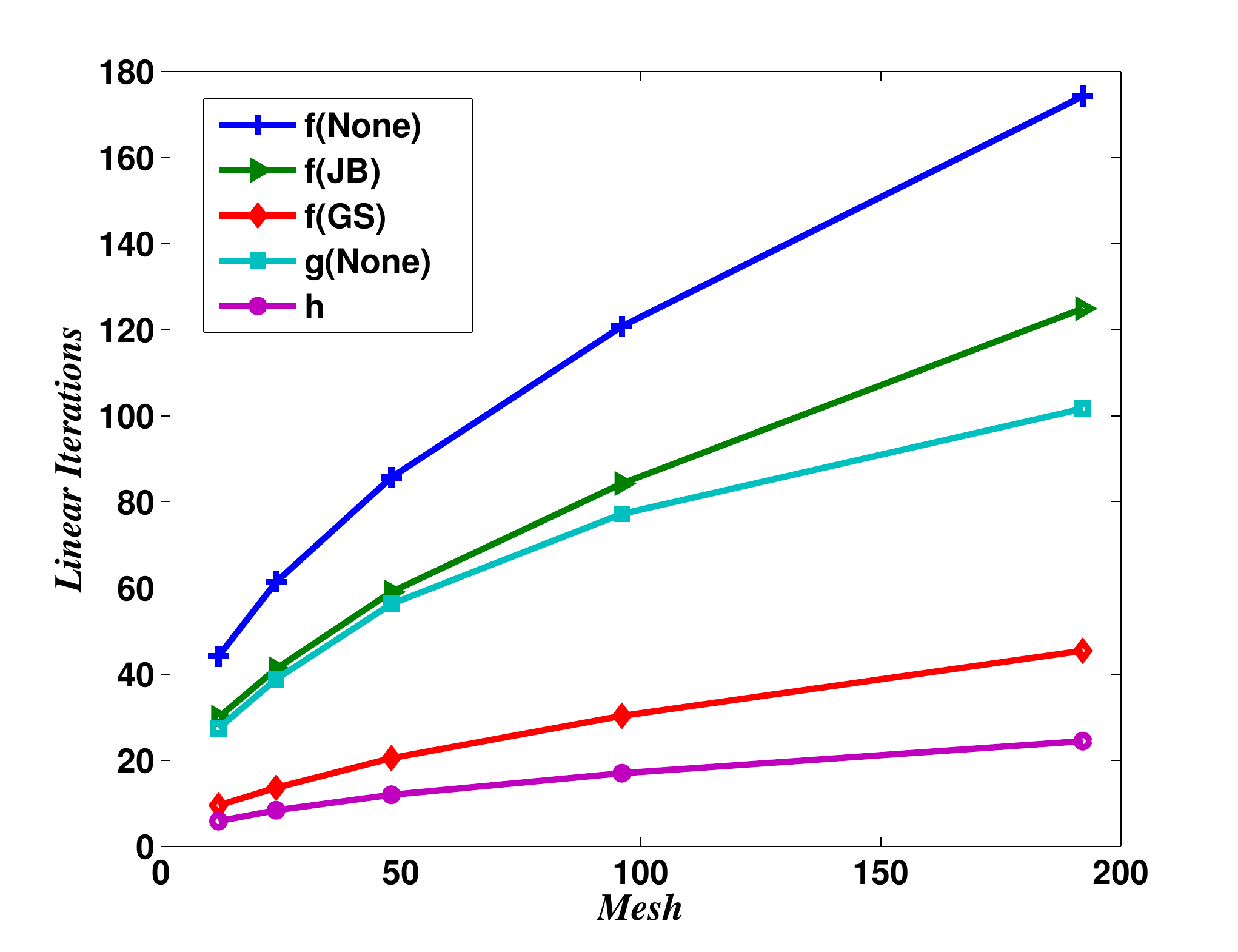}
\caption{Non-linear (left) and linear (right) iterations for different
preconditioning strategies}
\label{fig:4}
\end{center}
\end{figure}

The unpreconditioned method is unsurprisingly non scalable, at least
as far as the linear iterations are concerned. The number of
non-linear iterations grows only weakly with the number of mesh
points. The same is true for Jacobi preconditioning. 

Gauss--Seidel preconditioning, on the other hand, show only a modest
increase in the number of linear iterations, and a behavior for the
non-linear iterations that is in between that of the unpreconditioned
and the elimination strategies. 

As explained in section~\ref{sec:precond}, a limitation of the
elimination strategy is that it requires an exact solution of the
transport step. In this case, the Gauss--Seidel preconditioner might
prove useful: replacing the transport solve step by an approximation,
such as several iterations of a multi-grid solver, should lead to a
more efficient solution method, with similar convergence behavior. We
plan to explore this strategy in a forthcoming work.

Last, figure~\ref{fig:5} shows the time required by the various
methods. Since the cost of the methods is comparable, the ordering is
the same as that in the previous figure. It confirms the good
efficiency of the elimination strategy, with Gauss--Seidel
preconditioning as a distant second.

\begin{figure}[htbp]
\begin{center}
\includegraphics[width=0.52\textwidth]{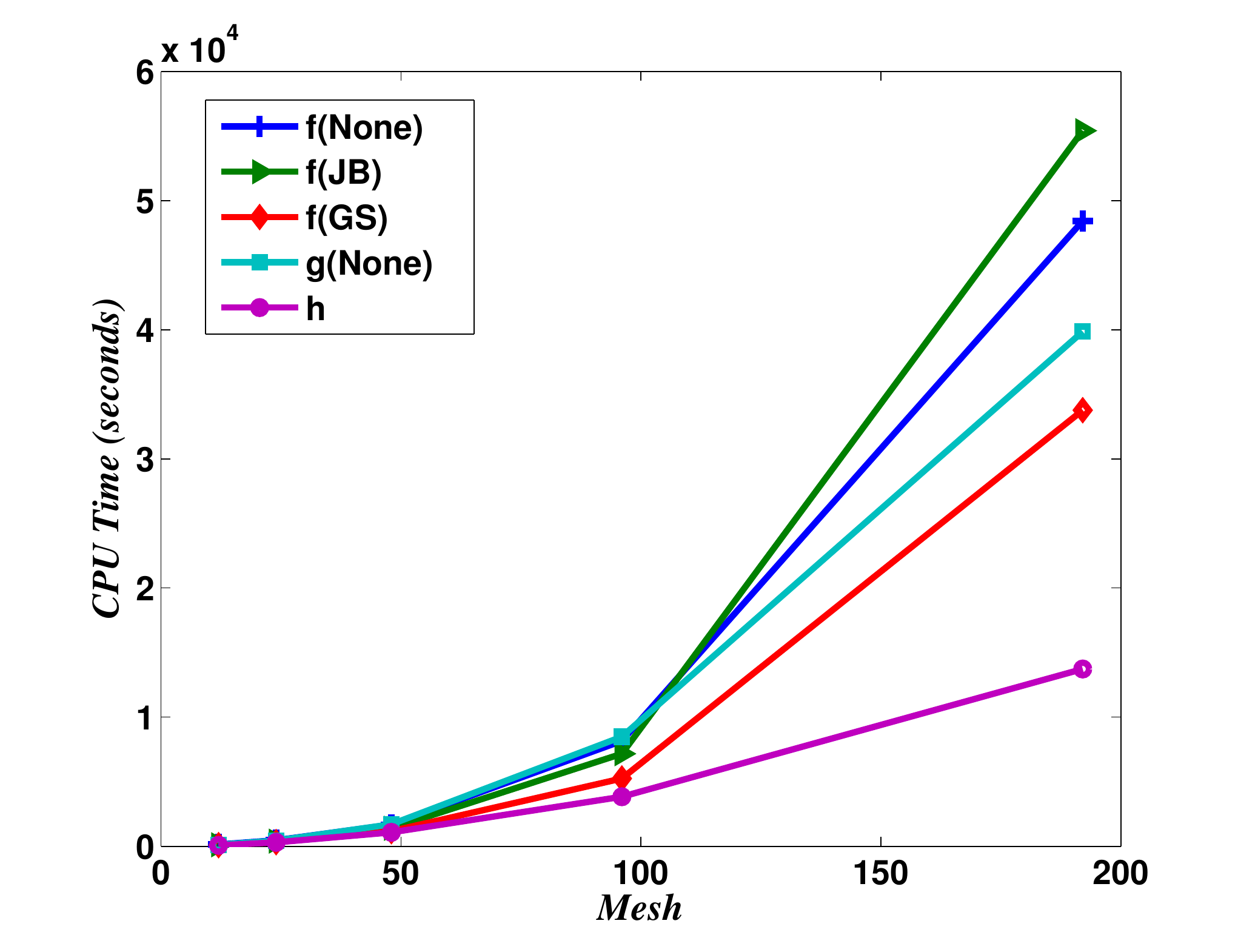}
\caption{CPU time as a function of the mesh size for the various
  solution methods}
\label{fig:5}
\end{center}
\end{figure}

\medskip
We can try and summarize the relative performances of the various methods as follows:
\begin{itemize}
\item the original $f$ formulation is not numerically scalable,
  neither  at the non-linear level, nor at the linear level;
\item the block Jacobi preconditioner applied to $f$ does not bring any
  improvement;
\item the block Gauss--Seidel preconditioner improves the linear
  performance, but does have a significant non-linear effect (nor was
  it expected);
\item the $ g$ formulation improves the linear performance, but not as
  much as Gauss--Seidel preconditioning;
\item the $h$ formulation, after elimination of the $C$ unknowns is the
  only method that gives a convergence independent of the mesh size. 
\end{itemize}

Of course, these conclusions are more or less natural: methods $f$ and
$g$ keep the ill-conditioning from the second order operator (except
for $g$ at the linear level). The elimination method, leads to a
bounded operator on $L^2$ (at least formally), and is expected to give
mesh independent convergence.

This good performance of the elimination method, at least on the linear
level, can be confirmed by looking at the \emph{field of values} of
the matrix $J_h$.  
The field of values of a matrix is the subset of the complex
plane defined by
\begin{equation*}
  W(J_h) = \set{\dfrac{x^H J_h x}{x^H x}, x \in \C, x \neq 0 }.
\end{equation*}
It includes the eigenvalues and is a convex set. For a non-symmetric
matrix, the convergence of GMRES is better described by the field of
values than by the eigenvalues (see~\cite{doi:10.1137/100806485}
or~\cite{Klawonn1999}). We conjecture that the field of values of the
Schur complement $J_h$  can be
bounded away from 0, independently of $N_x$. Though we have no proof
at the moment, we have the following numerical confirmation on
figure~\ref{fig:fovJ}, which shows the field of values \MK{and
  isolines of the $\epsilon$-pseudo-spectra} of the Jacobian matrix
$J_h$ for 2 different mesh sizes. One can see that the convex hull of
the field of values is indeed approximately independent of the mesh
size.  The figure was obtained thanks to the Eigtool
software~\cite{eigtool:02,doi:10.1137/S106482750037322X}.
\begin{figure}[htb]
  \centering
  \includegraphics[width=0.45\textwidth]{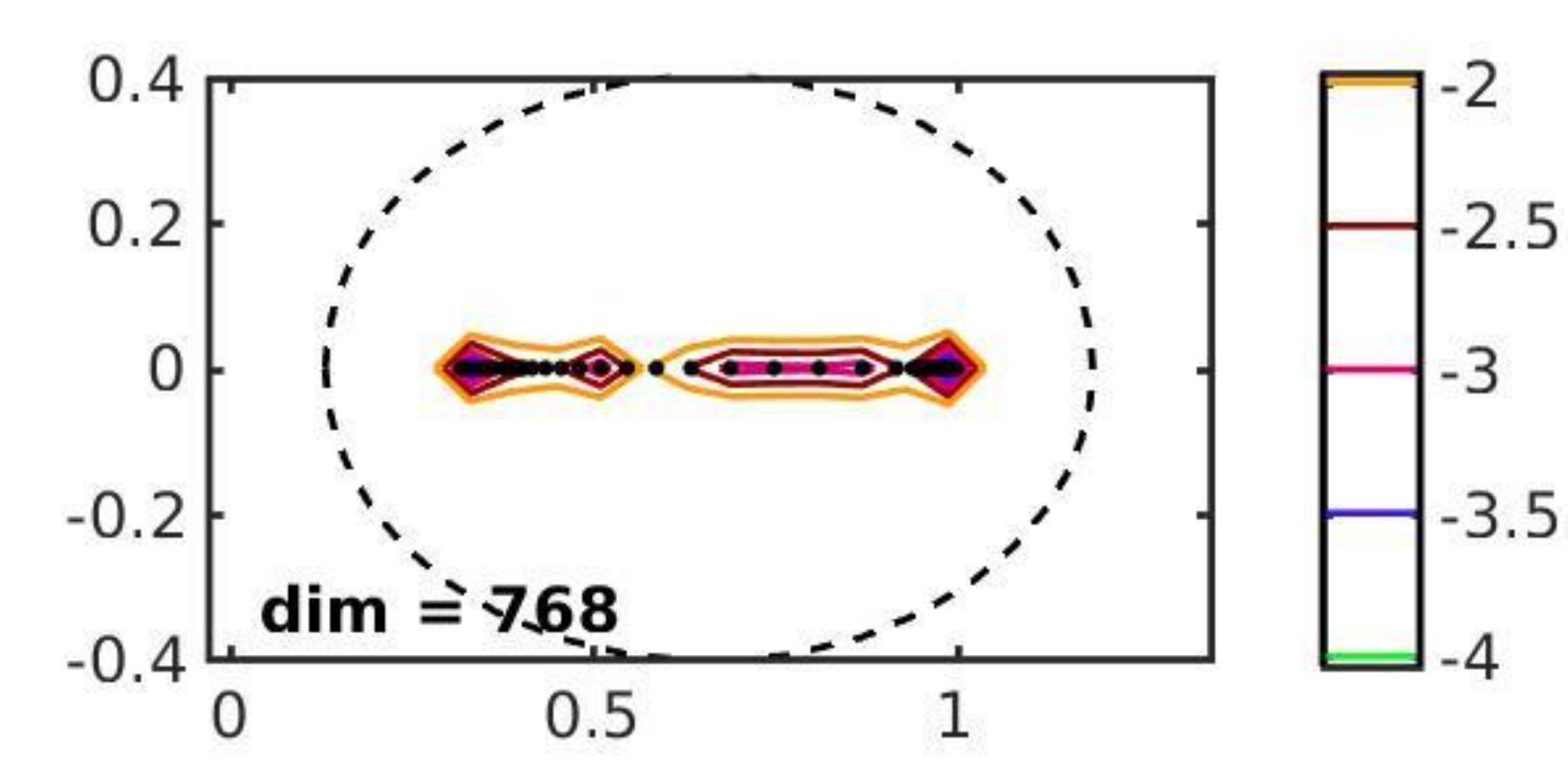}
  \includegraphics[width=0.45\textwidth]{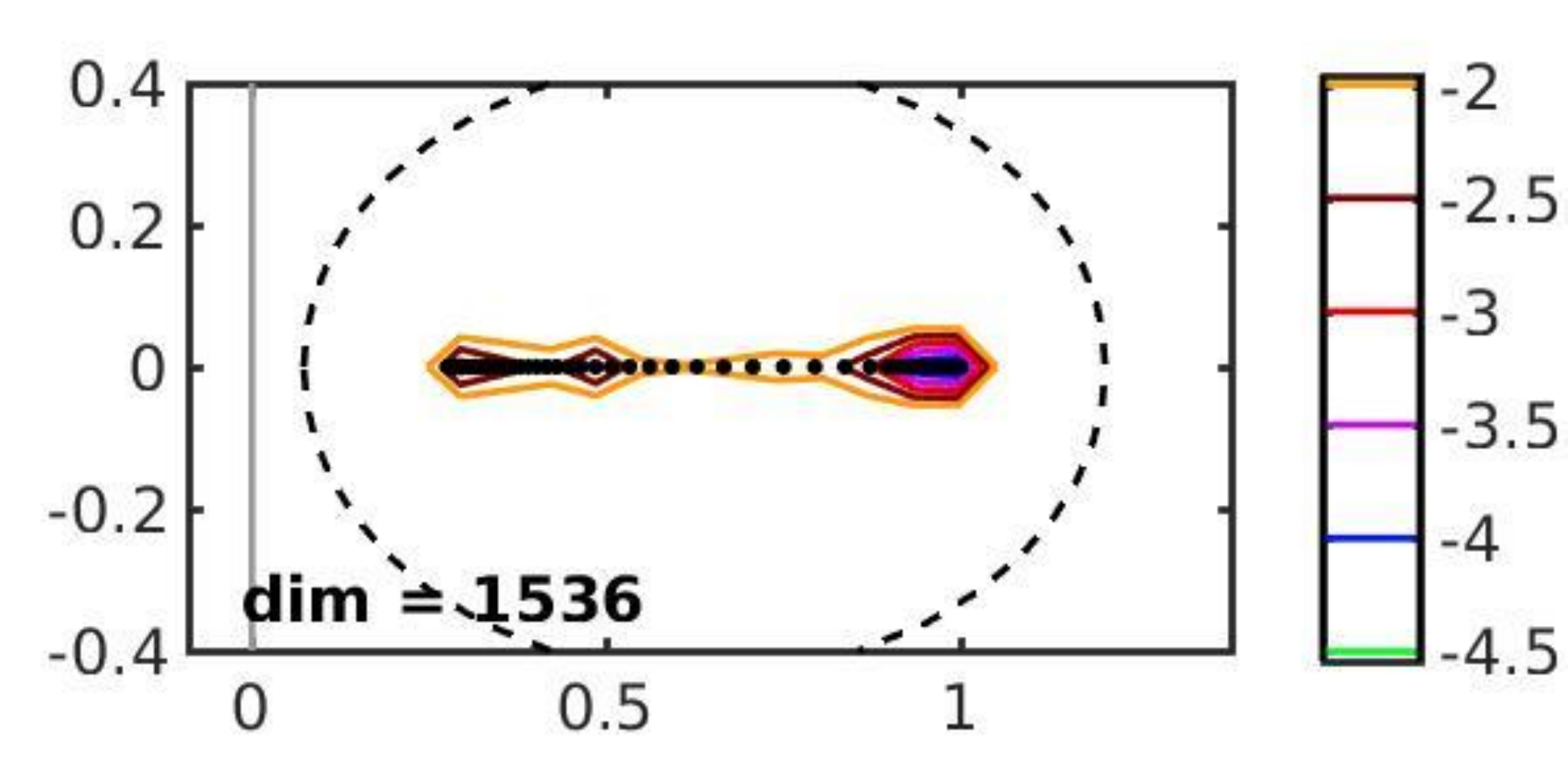}
  \caption{Field of values \MK{(dashed line) and pseudo-spectra (color) for matrix $J_h$}, with $N_x=768$ and $N_x=1536$}
  \label{fig:fovJ}
\end{figure}

\MK{
\subsection{Test 2. Ions exchange in a natural system}

Our second test case comes for a field study that includes
experimental results (see Valocchi et al~\cite{WRCR:WRCR3036}). 
 We follow the setup given in Fahs et al~\cite{Fahs2008}, as this reference includes more details for the numerical simulation. 
In this test case, four aqueous components ($\text{Na}^+$, $\text{Ca}^{2+}$, $\text{Mg}^{2+}$ and $\text{Cl}^-$) are injected into an homogeneous
landfill. During the transport, the aqueous components
react with the ion exchange sites of the soil (S). Three
reactions of ion exchange occur and lead to three
adsorbed species $\text{S}$--$\text{Na}$, $\text{S2}$--$\text{Ca}$, $\text{S2}$--$\text{Mg}$.\\
Chemical reactions, constants of
equilibrium, initial and boundary conditions are summarized in Table
~\ref{tab:3}.
\begin{table}[htb]
\begin{center}
\begin{tabular}{|c|cccc|c|c|}
\hline
 &  $\text{Na}^+$ &$\text{Ca}^{2+}$& $\text{Mg}^{2+}$& $\text{Cl}^-$& $\text{S}$ & log K\\
\hline
$\text{S}$--$\text{Na}$ & 1 & 0 & 0 & 0 & 1 & 4\\
$\text{S2}$--$\text{Ca}$& 0 & 1 & 0 & 0 & 2 & 8.602\\
$\text{S2}$--$\text{Mg}$ & 0 & 0 & 1 & 0 & 2 & 8.355\\
\hline
\hline
initial (mmol/l) & 248 & 165 & 168 & 161 & 750 &\\
injected (mmol/l) & 9.4& 2.12 & 0.494& 9.03 & &\\
\hline
\end{tabular}
\caption{Chemical reactions, initial and boundary conditions }
\label{tab:3}
\end{center}
\end{table}

In this test, we consider a column of length L=16 m. The values of the
transport parameters are given in Table ~\ref{tab:4}
\begin{table}[htb]
\begin{center}
\begin{tabular}{|l|l|}
\hline
Darcy  velocity   & u=0.2525 [m/h]\\
\hline
Dispersion
coefficient   & D=0.74235 [m2/h]\\
\hline
Porosity  & 0.35\\
\hline
\end{tabular}
\caption{Transport parameters}
\label{tab:4}
\end{center}
\end{table}
The simulated time is T=5000 h. The mesh size is equal to 0.08m , the
time step is equal to 0.11089 h corresponding to CFL coefficients of
1.

Figure~\ref{fig:ca-mg} shows the evolution of the concentrations
  in $\text{Ca}^{2+}$ and $\text{Mg}^{2+}$ at the end of the column, as a function
  of time (we use the same units as Fahs et al~\cite{Fahs2008} which
  are different than those originally used by Valocchi et
  al~\cite{WRCR:WRCR3036}).
\begin{figure}[htb]
  \centering
  \includegraphics[width=.45\textwidth]{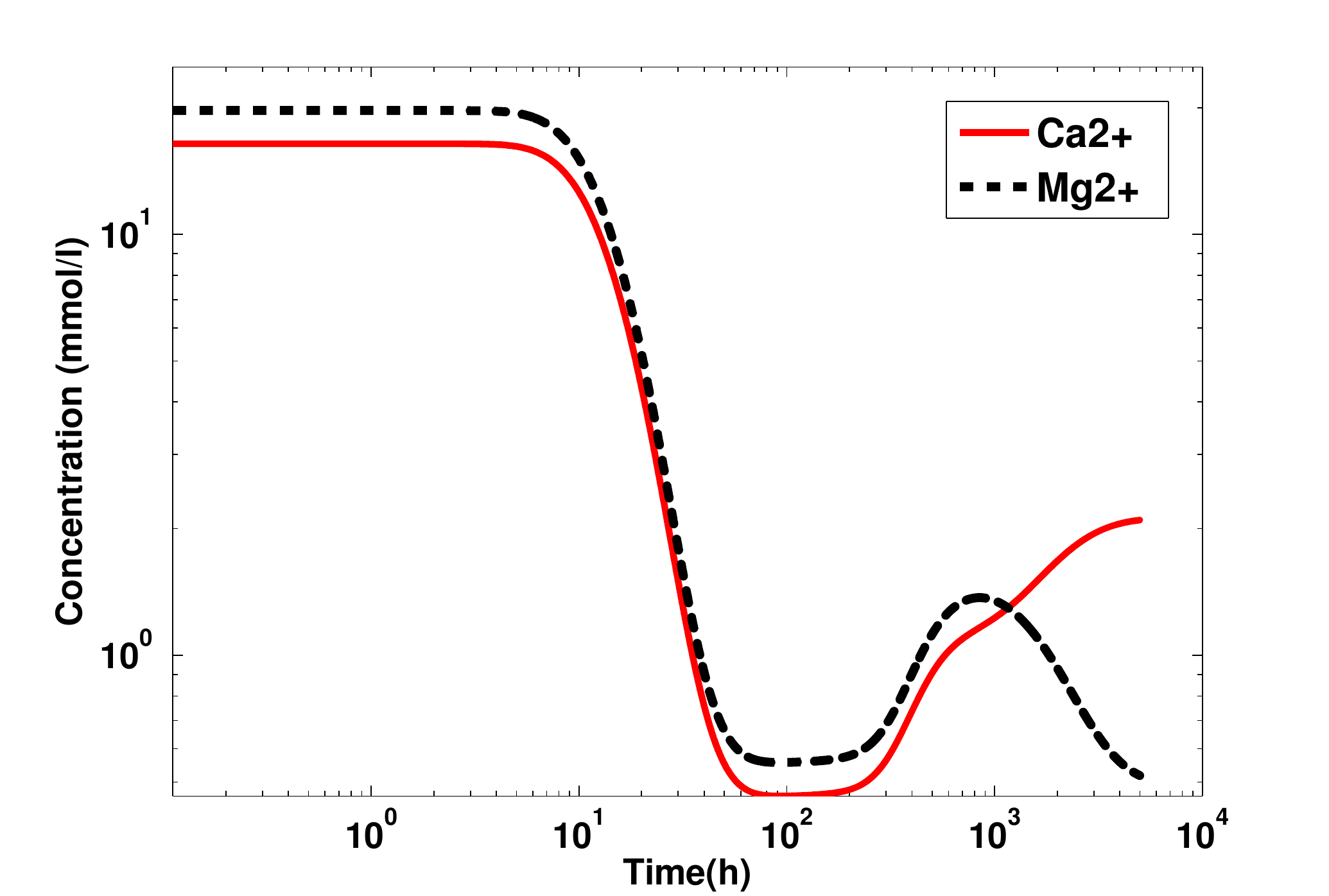}
  \caption{Variation of $\text{Ca}^{2+}$ and $\text{Mg}^{2+}$ concentrations as function of time in the output of the domain}
  \label{fig:ca-mg}
\end{figure}
The results are in qualitative agreement with those in both references 
(Figure 1 in Fahs et al~\cite{Fahs2008}, Figure 11 in Valocchi et al~\cite{WRCR:WRCR3036}), though it appears difficult to  make a precise 
comparison as the curves are in logarithmic scale on both axes.

Finally, figure~\ref{fig:landnliter_valo} compares the number of
linear and non-linear iterations for the various methods
presented. In this case, the number of non-linear iterations was very
close for all methods (with Block-Jacobi preconditioning a close
winner), and both the Gauss-Seidel preconditioning and the $h$-method again giving a convergence independent of
the mesh size 
\begin{figure}[htbp]
  \centering
    \includegraphics[width=.45\textwidth]{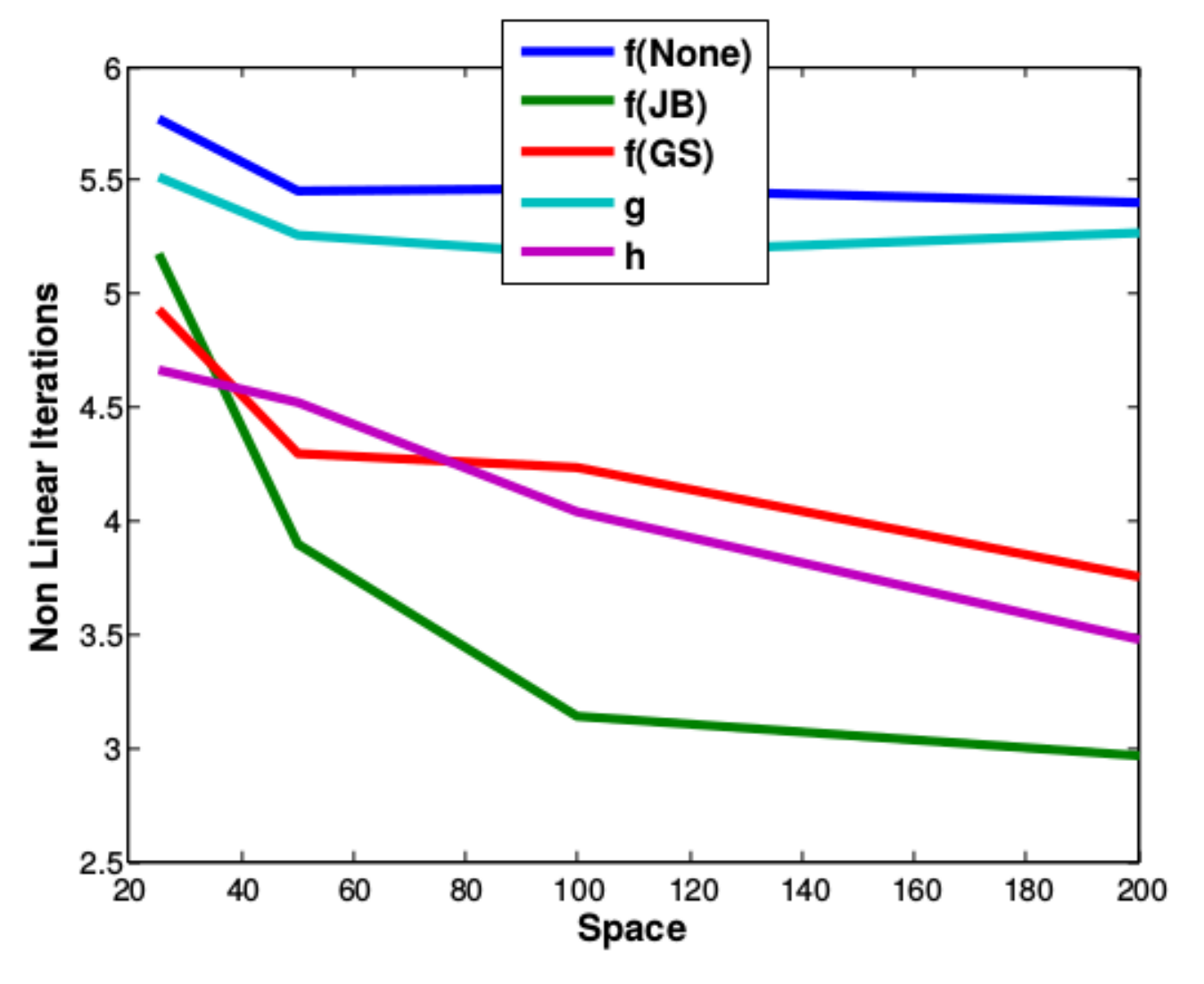}
    \includegraphics[width=.45\textwidth]{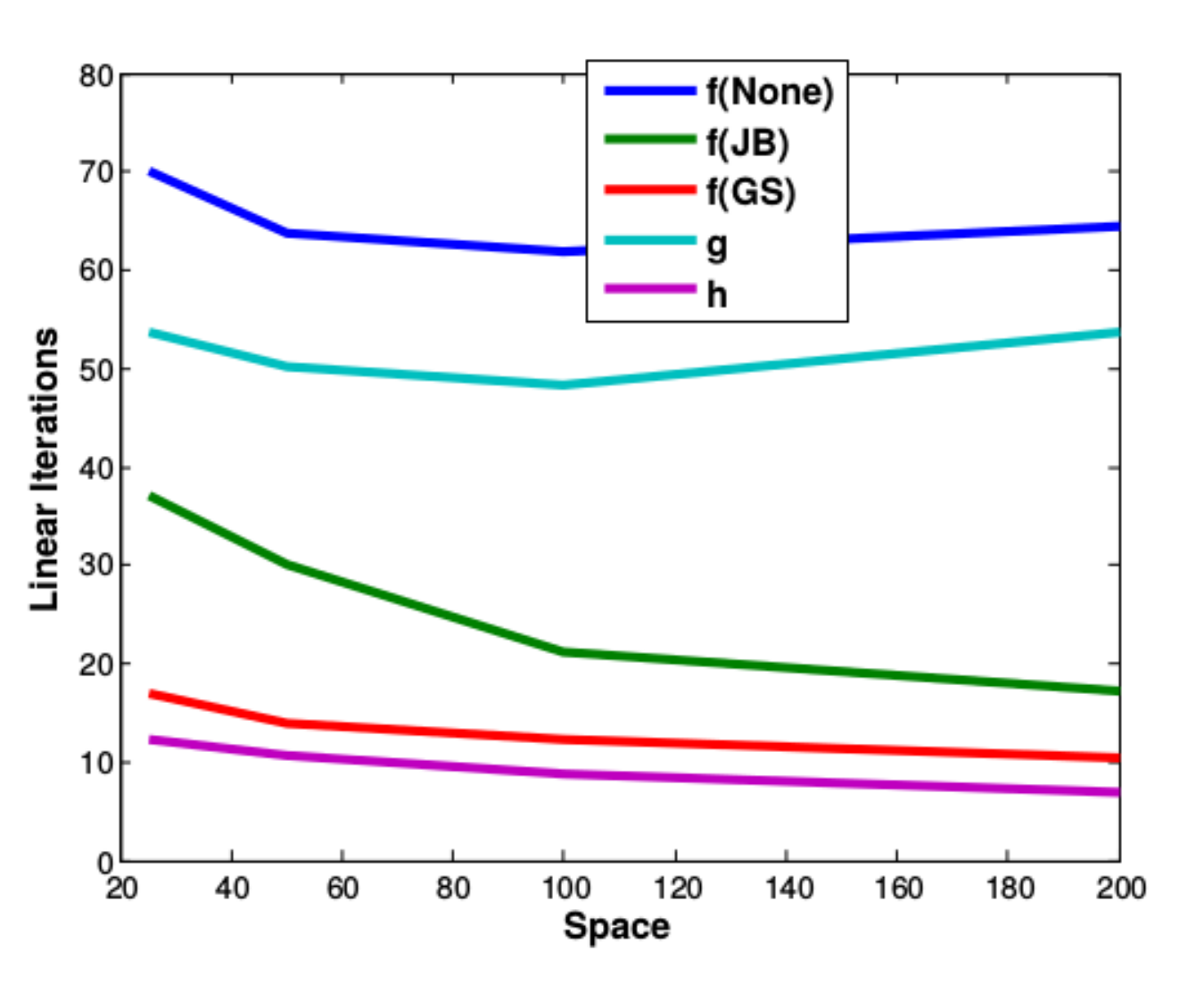}
  \caption{Number of non-linear (left) and linear (right) iterations for ion exchange
    example}
  \label{fig:landnliter_valo}
\end{figure}
}

  \section{Conclusion and perspectives}

In this work, several methods for improving the efficiency of a global 
approach for coupling transport and chemistry based on a Newton-Krylov
method were studied. 

 An alternative formulation and block preconditioners for linear system 
were used to accelerate the convergence of the Krylov method and to reduce CPU time.
The results show that the alternative formulation requires less CPU time than other preconditioners, and the
number of linear and non linear iterations becomes almost independent of the mesh. 

The reactive transport benchmark 1D problem proposed by GNR MoMaS was
used to demonstrate the efficiency of the method. 

Natural extensions of this work to multidimensional situations are
under way, as well as extensions to handle kinetic reactions.


\bibliography{ArticlePrecond}
\bibliographystyle{siam}

\end{document}